\documentclass[reqno,  a4paper]{amsart}

\usepackage[notref,notcite]{showkeys}
\usepackage{amsfonts}
\usepackage{color}
\usepackage{graphicx}
\usepackage{mathrsfs}
\usepackage{enumerate}
\usepackage{color}
\usepackage[utf8]{inputenc}
\usepackage[bookmarksnumbered,colorlinks, linkcolor=blue, citecolor=red, pagebackref, bookmarks, breaklinks]{hyperref}
\usepackage[left=.9in,right=1in,top=.8in,bottom=.8in]{geometry}
\newtheorem{thm}{Theorem}[section]

\newtheorem{lemma}[thm]{Lemma}
\newtheorem{prp}[thm]{Proposition}
\theoremstyle{definition}
\newtheorem{definition}[thm]{Definition}

\theoremstyle{remark}

\numberwithin{equation}{section}
\newcommand{\R}{\mathbb R}

\newcommand{\RR}{\mathbb{R}}

\newcommand{\ol}{\overline}
\newcommand{\Om}{\Omega}

\newcommand{\e}{\epsilon}
\newcommand{\la}{\lambda}
\newcommand{\al}{\alpha}

\begin{document}

    \title[Double phase Robin problem]{Fractional double phase Robin  problem involving variable order-exponents and Logarithm-type nonlinearity}
    \author[R. Biswas]{Reshmi Biswas}
    \author[A. Bahrouni]{Anouar Bahrouni}
    \author[A. Fiscella]{Alessio Fiscella}
    \address[R. Biswas]{Mathematics Department, Indian Institute of technology Delhi, Hauz Khas, Delhi 110016, India}
    \address[A. Bahrouni]{Mathematics Department,  Faculty of Sciences, University of Monastir, 5019 Monastir, Tunisia}
    \address[A. Fiscella]{Dipartimento di Matematica e Applicazioni, Universit\`a degli Studi di Milano-Bicocca, Via Cozzi 55, Milano 20125, Italy}
    \email[R. Biswas]{reshmi15.biswas@gmail.com}
    \email[A. Bahrouni]{bahrounianouar@yahoo.fr}
    \email[A. Fiscella]{alessio.fiscella@unimib.it}

    \keywords {Variable-order fractional $p(\cdot)$-Laplacian, Double phase problem, Logarithmic nonlinearity, variational methods\\
        \hspace*{.3cm} {\it 2010 Mathematics Subject Classifications}:
        35R11, 35S15, 47G20, 47J30.}

    \begin{abstract}
        The paper deals with the  logarithmic fractional equations with variable exponents
\begin{equation*}
    \left\{
\begin{aligned}
                &(-\Delta)^{s_1(\cdot)}_{p_1(\cdot)}(u)+(-\Delta)^{s_2(\cdot)}_{p_2(\cdot)}(u)
                +|u|^{\overline{p}_1(x)-2}u+|u|^{\overline{p}_2(x)-2}u
                =\la b(x)|u|^{\al(x)-2}u\\
                &\qquad\qquad\qquad\qquad\qquad\qquad\qquad\qquad\quad\quad\,\,\,\,\,+ \mu a(x)|u|^{r(x)-2}u\log|u|+\mu
                c(x)|u|^{\eta(x)-2}u ,&&x\in \Om, \\
                &\mathcal{N}^{s_1(\cdot)}_{p_1(\cdot)}(u)+\mathcal{N}^{s_2(\cdot)}_{p_2(\cdot)}(u)
                +\beta(x)(|u|^{\overline{p}_1(x)-2}u+|u|^{\overline{p}_2(x)-2}u)
                =0, &&x\in
                \R^N\setminus\overline{\Omega},
\end{aligned}
\right.
\end{equation*}

        where  $(-\Delta)^{s_{i}(\cdot)}_{p_{i}(\cdot)}$ and $\mathcal{N}^{s_i(\cdot)}_{p_i(\cdot)}$ denote the variable
        $s_{i}(\cdot)$-order $p_{i}(\cdot)$-fractional Laplace operator and the nonlocal normal $p_{i}(\cdot)$-derivative of $s_{i}(\cdot)$-order, respectively,
        with $s_{i}(\cdot):\mathbb R^{2N}\to(0,1)$ and $p_{i}(\cdot):\mathbb
        R^{2N}\to(1,\infty)$ ($i\in \{1,2\}$) being continuous. Here
        $\Omega\subset\mathbb{R}^{N}$ is a bounded smooth domain with
        $N>p_{i}(x,y)s_{i}(x,y)$ ($i\in \{1,2\}$) for any
        $(x,y)\in\overline{\Omega}\times\overline{\Omega}$, $\lambda$ and $\mu$
        are a positive parameters, $r(\cdot)$ and $\eta(\cdot)$ are two
        continuous  functions,
        while variable exponent $\alpha(x)$ can be close to the critical exponent
        $p^{*}_{2s_{2}}(x)=N\overline{p}_{2}(x)/(N-\overline{s}_{2}(x)\overline{p}_{2}(x))$,
        given with $\overline{p}_{2}(x)=p_{2}(x,x)$ and $\overline{s}_{2}(x)=s_{2}(x,x)$ for $x\in\overline{\Omega}$.
        Precisely, we consider two cases. In the first case, we deal with subcritical nonlinearity, that is,
        $\alpha(x)<p^{*}_{2_{s_{2}}}(x)$, for any $x\in \overline{\Omega}$.
        In the second case, we study the critical exponent, namely,
        $\alpha(x)=p^{*}_{2_{s_{2}}}(x)$
        for some $x\in \overline{\Omega}$. Then, using variational methods, we prove the
        existence and multiplicity of solutions and existence of ground state solutions to the above problem.
    \end{abstract}

    \maketitle
    \tableofcontents

    \section{Introduction}

    In this paper, we study the existence and the multiplicity of solutions for fractional $p_1(\cdot)\& p_2(\cdot)$-Laplacian problems with the nonlocal Robin boundary condition and involving a logarithmic type nonlinearities. More precisely, we first consider the following problem
\begin{equation}\label{eq1}
\left\{
\begin{aligned}
            &\mathcal{L}_{p_1,p_2}^{s_1,s_2} (u)
            +|u|^{\overline{p}_1(x)-2}u+|u|^{\overline{p}_2(x)-2}u
            =\la b(x)|u|^{\al(x)-2}u+  a(x)|u|^{r(x)-2}u\log|u|,
            &&x\in \Om, \\
            &\mathcal{N}^{s_1,s_2}_{p_1,p_2}(u)+\beta(x)(|u|^{\overline{p}_1(x)-2}u+|u|^{\overline{p}_2(x)-2}u)
            =0,&& x\in\R^N\setminus\overline{\Omega},
\end{aligned}
\right.
\end{equation}
    where $\Omega\subset\R^N$ is a bounded smooth domain with dimension $N\geq2$, $\la$ is a positive parameter, $a$, $b$ and $\beta$ are suitable non-negative functions, $s_1$, $s_2$, $p_1$, $p_2$ are uniformly continuous functions such that $\ol{p}_i(x):=p_i(x,x)$ and $\ol{s}_i(x):=s_i(x,x)$, for $i\in\{1,2\}$, with appropriate assumptions described later. The main operator in \eqref{eq1} is given by
    $$
    \mathcal{L}_{p_1,p_2}^{s_1,s_2} (u):=(-\Delta)^{s_1(\cdot)}_{p_1(\cdot)}(u)+(-\Delta)^{s_2(\cdot)}_{p_2(\cdot)}(u),
    $$
    where for any continuous variable exponent $p_i$
the fractional $p_i(\cdot)$-Laplacian operator
$(-\Delta)^{s_i(\cdot)}_{p_i(\cdot)}$, with $0<s_i(\cdot)<1<p_i(\cdot)$, is defined as
$$
(-\Delta)^{s_i(\cdot)}_{p_i(\cdot)}u(x)=P.V.\int_{\mathbb{R}^{N}}\frac{|u(x)-u(y)|^{s_i(x,y)-2}(u(x)-u(y))}{|x-y|^{N+p_i(x,y)s_i(x,y)}}dy,
\quad x\in\mathbb{R}^{N},
$$
along any $u\in C_0^\infty(\mathbb R^N)$, where P.V. denotes
the Cauchy principle value. Furthermore, in \eqref{eq1} we have a double phase Neumann boundary condition with
    $$
    \mathcal{N}^{s_1,s_2}_{p_1,p_2}(u):=\mathcal{N}^{s_1(\cdot)}_{p_1(\cdot)}(u)+\mathcal{N}^{s_2(\cdot)}_{p_2(\cdot)}(u)
    $$
where $\mathcal{N}^{s_i(\cdot)}_{p_i(\cdot)}$ denotes the nonlocal normal $p_i(\cdot)$-derivative of $s_i(\cdot)$-order, set as
    \begin{align*}
        \mathcal{N}^{s_i(\cdot)}_{p_i(\cdot)}u(x)=\int_{\Omega}\frac{|u(x)-u(y)|^{p_i(x,y)-2}(u(x)-u(y))}{|x-y|^{N+s_i(x,y)p_i(x,y)}}\,dy \quad \text{for }x\in\R^N\setminus\overline{\Omega}.
    \end{align*}

    The nonlocal operators  are nowadays gaining popularity in applied sciences, in
    theoretical studies and also in real world applications. Precisely,
    fractional Sobolev spaces are well known since the beginning of the
    last century, especially in the framework of harmonic analysis. The
    starting point in the study of fractional problems is due to the
    pioneering papers of  of Caffarelli et al. \cite{11,12,13}. Based on
    this, several other works have been published in the nonlocal
    framework, see for instance, \cite{10,14,23,22,25}.  For a
    comprehensive introduction to the study of fractional equations and
    the use of variational methods in the treatment of these problems,
    we refer to the monograph by  Di Nezza, Palatucci and Valdinoci
    \cite{23}.\\

    In this paper, we study the variable-order fractional
    $p(\cdot)$-Laplace operator in combination with a logarithmic
    nonlinearity. To our best knowledge, Kaufmann et al. \cite{kaufmann}
    first introduced some results on fractional Sobolev spaces with
    variable exponent  and the corresponding
    fractional $p(\cdot)$-Laplacian for constant order $0<s<1$. There, the authors established
    compact embedding theorems of these spaces
    into variable exponent Lebesgue spaces.
    In \cite{6}, Bahrouni and R\u{a}dulescu  obtained some further qualitative properties of the fractional Sobolev spaces with variable exponents  and the fractional
    $p(\cdot)$-Laplacian. After that, by imposing variable growth on the order $s$, Biswas and Tiwari \cite{rs} introduced variable-order fractional $p(\cdot)$-Laplacian and
    corresponding variable order fractional Sobolev spaces with variable exponent.
    For further studies on this kind of
    problems with some different approaches, see \cite{5,rs-1,rs2,ky-ho}, where the authors studied the problems with homogeneous Dirichlet boundary datum, that is,   $u=0$ in $\RR^N\setminus\Om$. In \cite{BRW}, Bahrouni et al. introduced the notion of Robin boundary condition  in the variable exponent nonlocal framework.
    There are only two papers in the literature, as per as we are concerned, regarding   critical  nonlocal problems
    involving variable order, variable exponents.  We refer to the work by Ho and Kim \cite{Ho-Conc}, where the authors first introduced  a critical embedding and
    concentration-compactness principles for such fractional Sobolev spaces with variable exponents and constant order $0<s<1$. Later, Fiscella et al. \cite{zuo} extended these results regarding critical  nonlocal problems with variable exponent in the variable order set up. This paper further extends the works as in \cite{BRW,zuo}.
    \\

      The main operator in \eqref{eq1} is a generalized version of the constant $p_1\&p_2$-Laplacian which has a wide amount of applications in  plasma physics, bio-physics, reaction-diffusion and others, as shown in \cite{1.0, 1.1, 1.2}. When the exponents $p_i$, $s_i$ are constants, for $i\in\{1,2\}$,  we refer
     to the survey article \cite{21}, also to \cite{d3,PRR,SRRZ} for the local case, that is $s_i\to 1^-$, and \cite{amb,ambrosioR,d1,d2} for the nonlocal case.
     But  \eqref{eq1} has a more delicate structure, due to the non-homogeneity, as well as  the variable growth. For the local case with variable exponents, that is $s_i\to1^-$, we refer to the works in \cite{Bahrouni-Radulescu-Repovs, Radulescu}.  Among few recent works dealing with these problems in nonlocal set up, that is, variable-order fractional $p_1(\cdot)\&p_2(\cdot)$-Laplacian equations,  one can see the references {\cite{tong,zfb1} for the problem involving  homogeneous Dirichlet boundary datum and \cite{rsk} for the problem involving Robin boundary condition. \\

    On the other hand, the problems with the presence of logarithmic
    nonlinearity have been studied by several authors. In particular in
    \cite{tian}, Tian proved that the following problem
\begin{equation*}
\left\{
\begin{aligned}
&-\Delta u=a(x)u\log |u| && \quad \text{in} \quad    \Omega, \\
&u=0&& \quad \text{on} \quad  \partial \Omega,
\end{aligned}
\right.
    \end{equation*}
    has at least two nontrivial solutions provided that $a(\cdot)$
    changes sign on $\Omega$.  In \cite{d'avenia}, d'Avenia et al.
    considered the following fractional logarithmic Schr\"odinger
    equation
    \begin{equation*}
        (-\Delta)^{s} u+Wu=u \log |u|^{2},\qquad x\in \mathbb{R}^{N},
    \end{equation*}
    where $W:\mathbb R^N\to \RR^+$ is a continuous function . By employing the fractional logarithmic Sobolev
    inequality, the authors obtained the existence of infinitely many
    solutions.  Next,  in \cite{truong}, Truong studied the following
    problem fractional $p$-Laplacian equations with logarithmic
    nonlinearity
    \begin{equation*}
        (-\Delta)_{p}^{s} u+V(x)|u|^{p-2}=\lambda l(x) |u|^{p-2}u \log |u|,\qquad
        x\in \mathbb{R}^{N},
    \end{equation*}
    where $l(\cdot):\mathbb R^N\to \RR$ is a sign-changing weight function. Using the Nehari manifold
    approach, the author proved the existence of at least two nontrivial
    solutions. In this article, we generalizes all these  aforementioned works involving logarithmic nonlinearity, in the variable order-exponent framework with the nonhomogeneous double phase operator. Hence, it calls for carrying out more delicate analysis as compared to that in the previous mentioned literature.  \\

    Inspired by all the above works, for the first time in literature we study fractional Laplacian problems  with variable-order and variable-exponent, involving logarithmic nonlinearities, such as \eqref{eq1}.
    For this, we fix some notations as follows.
    For any set $\mathcal{D}$ and any function $\Phi:\mathcal{D}\rightarrow\mathbb R$, we denote
    {\begin{align*}
            \Phi^{-}:=\inf_{\mathcal{D}} \Phi(x)\quad\text{and}\quad \Phi^{+}:=\sup_{ \mathcal{D}}\Phi(x).
    \end{align*}}
    \noindent We  define the function space
    $$C_+(\mathcal{D}):=\left\{\Phi:\mathcal{D}\to \R \text {~is uniformly continuous~} \big|~ 1 <\Phi^{-}\leq \Phi^{+}<\infty\right\}.$$
    We consider the following standard hypotheses throughout this paper:
    \begin{itemize}
        \item[{$(H_1)$:}] \textit{$s_i:\mathbb R^{2N}\rightarrow(0,1)$ are uniformly continuous and symmetric functions, i.e., $s_i(x,y)=s_i(y,x)$ for all $(x,y)\in \mathbb R^{2N}$ and for all $i\in \{1,2\}$,
        with  $0<s_1^-\leq s_1^+<s_2^-\leq s_2^+<1.$}
    \end{itemize}
    \begin{itemize}
        \item[{$(H_2)$:}] \textit{$p_i\in C_+(\mathbb R^{2N})$ are symmetric functions, i.e., $p_i(x,y)=p_i(y,x)$ for all $(x,y)\in \mathbb R^{2N}$ and for all $i\in \{1,2\}$, with  $1<p_1^-\leq p_1^+<p_2^-\leq p_2^+<\infty$ and such that $s_2^+p_2^+<N$.}
    \end{itemize}
    \begin{itemize}
        \item[{$(H_3)$:}] \textit{$\beta:\mathbb R^N\setminus\Omega\to\mathbb R$ is a non-negative function and $\beta\in L^\infty(\mathbb R^N\setminus\Omega)$.}
    \end{itemize}
    \begin{itemize} \item[$(H_4)$:]  \textit{$a(\cdot):\overline\Om\to\mathbb R $ is a non-negative function and  $a \in L^{\infty}(\Om)$.}
    \end{itemize}
    \begin{itemize}
        \item[$(H_5)$:] \textit{$b(\cdot):\overline\Om\to\mathbb R $ is a positive function and $b\in L^{g(\cdot)}(\Om),$ where $g\in C_+(\ol\Om)$ and $g(x)=\frac{r(x)}{r(x)-\al(x)}$.}
    \end{itemize}

    For studying  problem \eqref{eq1}, we need also the following information about the subcritical exponents $\alpha(\cdot)$ and $r(\cdot)$:
    \begin{itemize}
        \item[$(H_6)$:] \textit{$\al(\cdot)$, $r(\cdot)\in C_+(\Om)$ such that $$1<\al^-\leq\al^+<p_1^-< 2p_2^+<r^-\leq r^+<p_{2_{s_2}}^*(x):=\frac{N\ol p_2(x)}{N-\ol s_2(x)\ol p_2}$$
        for all $x\in\Omega$, where $\ol p_2(x)=p_2(x,x)$ and $\ol s_2(x)=s_2(x,x)$.
        Moreover, we assume that there exist $x_0\in \Omega$ and
        $R_{0}>0$ such that
        $$r(x)=r^{+}, \quad \mbox{for all }x\in B^{c}(x_0,R_0).$$}
    \end{itemize}

    Thus, we are ready to
    introduce our first two main results.
    \begin{thm}\label{mainthm}
        Let $\Om$ be a smooth bounded domain in $\RR^N$ and let  $(H_1)$-$(H_6)$ hold.
        Then there exists a constant $\Lambda:=\Lambda(s_i,p_i,\Om,N,\al,\beta,r)>0$ such that for any $\la<\Lambda$ problem \eqref{eq1} achieves at least two distinct nontrivial weak solutions.
    \end{thm}
The proof of Theorem \ref{mainthm} is based on a combination of the mountain pass theorem and a minimization argument. In this direction, we strongly need $(H_6)$ to handle the logarithmic nonlinearity. In order to get the geometry of the functional related to \eqref{eq1}, we have to be aware that the logarithmic term can assume negative value. Indeed, the dominant exponent in \eqref{eq1} is $r(\cdot)$, as stated in $(H_6)$, which also involves the logarithmic nonlinearity.

        After Theorem \ref{mainthm}, we look for ground state solutions for \eqref{eq1}. For this, we need a new structural assumption for $a(\cdot)$:
   \begin{itemize} \item[$(H_{4}^{'})$:]  \textit{$a(\cdot):\overline\Om\to\mathbb R $ is a non-positive function and  $a \in L^{\infty}(\Om)$.}
    \end{itemize}
        Then, by the new geometry of \eqref{eq1} we are able to state the following existence result.
    \begin{thm}\label{mainthm-2}
        Let $\Om$ be a bounded domain in $\RR^N$ and let  $(H_1)$-$(H_3)$, $(H_{4}^{'})$ and $(H_5)-(H_6)$
        hold.
        Then, for any $\la>0$ problem \eqref{eq1} achieves a ground state
        solution.
    \end{thm}
    In the second part of the paper we discuss the critical case. For this, we consider the following problem
    \begin{equation}\label{eq2}
  \left\{
    \begin{aligned}
            &\mathcal{L}_{p_1,p_2}^{s_1,s_2} (u)
            +|u|^{\overline{p}_1(x)-2}u+|u|^{\overline{p}_2(x)-2}u
            = b(x)|u|^{\al(x)-2}u+ \mu a(x)|u|^{r(x)-2}u\log|u|\\
            &\qquad\qquad\qquad\qquad\qquad\qquad\qquad\qquad\quad+\mu c(x)|u|^{\eta(x)-2}u,&&
            x\in \Om, \\
            &\mathcal{N}^{s_1,s_2}_{p_1,p_2}(u)+\beta(x)(|u|^{\overline{p}_1(x)-2}u+|u|^{\overline{p}_2(x)-2}u)
            =0,&& x\in\R^N\setminus\overline{\Omega},
\end{aligned}
\right.
\end{equation}
where $\mu>0$ is a positive parameter and the structural assumptions assumed for \eqref{eq1} still hold true. Besides $(H_1)$-$(H_5)$, in order to deal with a critical Sobolev situation as similarly done in \cite{zuo}, we need the following new hypotheses for $p_2(\cdot)$, $b(\cdot)$, $\alpha(\cdot)$ and $r(\cdot)$:
    \begin{itemize}
        \item[$(H_{5}^{'})$:] \textit{$b(\cdot):\overline\Om\to\mathbb R $ is a positive function and $b\in L^{\infty}(\Om)$.}
    \end{itemize}
    \begin{itemize} \item[$(H_{7})$:]  \textit{$p_2(\cdot)$ satisfies the following
            log-H\"{o}lder type continuity condition
            $$
            \inf\limits_{\varepsilon>0}\sup\limits_{(x,y)\in \Omega\times\Omega,
                ~0<|x-y|<1/2}\left|p_2(x,y)-p^{-}_{2,\Omega_{x,\varepsilon}\times\Omega_{y,\varepsilon}}\right|
            \log\frac{1}{|x-y|}<\infty,
            $$
            where $\Omega_{z,\varepsilon}=B_{\varepsilon}(z)\bigcap\Omega$ for
            $z\in\Omega$ and $\varepsilon>0$, and
            $p^{-}_{2,\Omega_{x,\varepsilon}\times\Omega_{y,\varepsilon}}=\inf\limits_{(x^{'},y^{'})\in\Omega_{x,\varepsilon}\times\Omega_{y,\varepsilon}}p_2(x^{'},y^{'})$.}
    \end{itemize}
    \begin{itemize} \item[$(H_{8})$:]
        \textit{$\alpha(\cdot)$, $r(\cdot):\overline{\Omega}\to\mathbb R$ are
        continuous functions verifying $1<r^{-}<r^{+}<\alpha
        ^{-}\leq \alpha^{+}$, and for any $x\in\overline{\Omega}$ there exists
            $\varepsilon=\varepsilon(x)>0$ such that
            $$
            \sup\limits_{y\in\Omega_{x,\varepsilon}}\alpha(y)\leq\frac{N\inf_{(y,z)\in\Omega_{x,\varepsilon}\times\Omega_{x,\varepsilon}}p_2(y,z)}{N-\inf_{(y,z)\in\Omega_{x,\varepsilon}\times\Omega_{x,\varepsilon}}s_2(y,z)\inf_{(y,z)\in\Omega_{x,\varepsilon}\times\Omega_{x,\varepsilon}}p_2(y,z)}
            $$
            with $\Omega_{x,\varepsilon}$ as defined in $(H_{7})$.}
    \end{itemize}
Assumption $(H_{8})$ somehow states the relation between $\al(x)$ and the critical exponent
            $p^*_{2s_2}(x)=N\overline{p}_2(x)/(N-\overline{s}_2(x)\overline{p}_2(x))$, where $\overline{p}_2(x)=p_2(x,x)$ and
            $\overline{s}_2(x)=s_2(x,x)$ for $x\in\overline{\Omega}$. Of course, $\al(x)\leq p^*_{2s_2}(x)$ but not necessarily $\al(x)=p^*_{2s_2}(x)$,
            for any $x\in\overline{\Omega}$. We obviously and tacitly assume along the paper that $\al(x)=p^*_{2s_2}(x)$ for some $x\in\overline{\Omega}$.

In order to control the logarithmic nonlinearity in \eqref{eq2}, with respect to \eqref{eq1} we need to add another nonlinearity with subcritical variable exponent $\eta(\cdot)$, assuming:
    \begin{itemize}
        \item[$(H_{9})$:]  $c(\cdot):\overline\Om\to\mathbb R $ is a non-negative function and  $c \in L^{\infty}(\Om)$.
    \end{itemize}
    \begin{itemize}
        \item[$(H_{10})$:] $\eta(\cdot)\in C_+(\Om)$ and there exists $\sigma_0>0$ such that $2p_{2}^{+}<\sigma_0<\eta^{-}\leq\eta^+<r
        ^-<\alpha^-$ and
        $$a(x)\leq e(r^--\eta^+)\min\left\{1, \frac{(\eta^{-}-\sigma_{0})r^{+}}{\eta^{-}(r^+-\sigma_{0})}\right\}c(x), \ \ \mbox{a.e} \ \ x\in \Omega.$$
    \end{itemize}
    Thus, our next result  for the critical case reads as follows.
    \begin{thm}\label{mainthm-3}
        Let $\Om$ be a bounded domain in $\RR^N$ and let  $(H_1)$-$(H_4)$, $(H_{5}^{'})$, $(H_{7})$-$(H_{10})$ hold. Then, there exists
        $\mu_{0}>0$ such that for any $\mu>\mu_0$ problem \eqref{eq2}
        admits at least one nontrivial weak solution $v_\mu$. Furthermore, we have
        \begin{equation}\label{asym}
        \lim_{\mu\to\infty}\|v_\mu\|=0.
        \end{equation}
    \end{thm}
    The symbol $\|\,\cdot\,\|$ in Theorem \ref{mainthm-3} denotes the norm of the solution space $X$, which will be introduced in Section \ref{sec3}. The proof of Theorem \ref{mainthm-3} is mainly based on the application of the mountain pass theorem. In this direction, assumption $(H_{10})$ plays a crucial role to verify the validity of the Palais Smale condition at the critical mountain pass level $c_\mu$. Indeed, in \eqref{eq2} we have to deal with a negative contribution of $\log|u|$, whenever $u$ is sufficiently small. For this, we need the relation between weight functions $a(\cdot)$ and $c(\cdot)$ given $(H_{10})$, which helps us to study the asymptotic behavior of $c_\mu$ whenever $\mu\to\infty$, useful also for proving \eqref{asym}.
Finally, Theorems \ref{mainthm}, \ref{mainthm-2} and \ref{mainthm-3} generalize in several directions \cite[Theorem 4.3]{BRW}.

    The paper is organized as follows. In Section \ref{sec2} we present the
    basic properties of the fractional Sobolev space with variable
    exponent. In Section \ref{sec3}, we introduce our solution space $X$ with some
    basic properties and crucial technical lemmas, and we state the
    variational formulation of \eqref{eq1} and \eqref{eq2}. Moreover, we  prove a result
    which provides an estimate for logarithmic nonlinearity.
    In Section \ref{sec4}, combining these abstract results with the
    variational arguments, we give the proofs of Theorem \ref{mainthm}
    and  Theorem \ref{mainthm-2}. In the last section, we deal with the critical case and establish the proof of Theorem \ref{mainthm-3}.

    \section{Preliminaries results}\label{sec2}

    \subsection{Variable exponent Lebesgue spaces}

    In this section first we recall some basic properties of the  variable exponent Lebesgue spaces, which we will use  to prove our main results.

    For $q(\cdot)\in C_+(\Om)$, we define the variable exponent Lebesgue space $L^{q(\cdot)}(\Om)$ as
    $$
    L^{q(\cdot)}(\Om) := \left\{ u : \Om\to\mathbb{R}\  \text{is~measurable}\,\bigg|\, \int_{\Om} |u(x)|^{q(x)} \;dx<\infty \right\}
    $$
    which is a separable, reflexive, uniformly convex Banach space, {see \cite{diening,fan}}, with respect to the Luxemburg norm
    $$
    \|{u}\|_{ L^{q(\cdot)}(\Om)}:=\inf\left\{\eta
    >0:\,\,\rho_{\Om}^q\left(\frac{u}{\eta}\right)\le1\right\},
    $$
    where $\rho_{\Om}^{q}:\ L^{q(\cdot)}(\Om)\to\mathbb{R}$ is the modular function set as
    $$
    \rho_{\Om}^q(u):=\int_{\Om }|u|^{q(x)}\; dx, \  \text{ for~all~} u\in L^{q(\cdot)}(\Om).
    $$
    Then, we have the following relation between the Luxemburg norm and the modular function.

    \begin{prp} \label{norm-mod}
        {\rm (\cite{fan})} Let $u\in L^{q(\cdot)}(\Om)$, $(u_n)_{n\geq1}\subset L^{q(\cdot)}(\Om)$ and $\eta>0$. Then:
        \begin{enumerate}
            \item[(i)]  If $u\neq0$, then $\eta=\|u\|_{ L^{q(\cdot)}(\Om)}$ if and only if  $\rho_{\Om}^q(\frac{u}{\eta})=1$;
            \item[(ii)] $\rho_{\Om}^q(u)>1$ $(=1;\ <1)$ if and only if  $\|u\|_{ L^{q(\cdot)}(\Om)}>1$ $(=1;\ <1)$,
            respectively;
            \item[(iii)] If $\|u\|_{ L^{q(\cdot)}(\Om)}>1$, then $
            \|u\|_{ L^{q(\cdot)}(\Om)}^{q^-}\le \rho_{\Om}^q(u)\le
            \|u\|_{L^{q(\cdot)}(\Om)}^{q^+}$;
            \item[(iv)] If $\|u\|_{ L^{q(\cdot)}(\Om)}<1$, then
            $\|u\|_{L^{q(\cdot)}(\Om)}^{q^+}\le \rho_{\Om}^q(u)\le
            \|u\|_{L^{q(\cdot)}(\Om)}^{q^-}$;
            \item[(v)] ${\displaystyle \lim_{n\to  \infty} }\| u_{n} - u \|_{ L^{q(\cdot)}(\Om)} =0$ if and only if ${\displaystyle \lim_{n\to  \infty}} \rho_{\Om}^q(u_{n} -u)=0.$
        \end{enumerate}
    \end{prp}

    Let $q'(\cdot)\in C_+(\Om)$ be the conjugate function of $q(\cdot),  $ that is  $1/q(x)+1/q'(x)=1$ for any $x\in\Omega$. Then, we have the following H\"{o}lder inequality for variable exponent Lebesgue spaces.

    \begin{lemma}{\rm (H\"{o}lder inequality)} \label{Holder}{\rm (\cite{fan})}
        For any $u\in
        L^{q(\cdot)}(\Om)$ and $v\in L^{q'(\cdot)}(\Om)$, we have
        {$$
            \Big|\int_{\Om} uv\,d x\Big|
            \leq
            2\|u\|_{ L^{q(\cdot)}(\Om)}\|{v}\|_{
                L^{q'(\cdot)}(\Om)}.
            $$}
    \end{lemma}

    \begin{lemma}{\rm(\cite{sweta})}\label{lemA1}
        Let  $\vartheta_1\in L^\infty(\Om)$ such that $\vartheta_1\geq0$ and $\vartheta_1\not\equiv 0$ a.e. in $\Omega$. Let $\vartheta_2:\Om\to \RR$ be a measurable function      such that $\vartheta_1\vartheta_2\geq 1$ a.e. in $\Om.$ Then for every $u\in L^{\vartheta_1(\cdot)\vartheta_2(\cdot)}(\Om),$
        $$
        \left\| |u|^{\vartheta_1(\cdot)}\right\|_{L^{\vartheta_2(x)}(\Om)}\leq\|u\|_{L^{\vartheta_1\vartheta_2(\cdot)}(\Om)}^{\vartheta_1^-}+\| u\|_{L^{\vartheta_1\vartheta_2(\cdot)}(\Om)}^{\vartheta_1^+}.
        $$
    \end{lemma}

    \subsection{Variable order fractional Sobolev spaces with variable exponents}

    Next, we define the fractional Sobolev spaces with variable order and variable exponents, as in {\cite{rs}}. We set
    \begin{align*}
        W&=W^{s_2(\cdot),p_2(\cdot)}(\Om)\nonumber\\ &:=\left\{ u\in L^{\overline{p}_2(\cdot)}(\Om):
        \int_{\Om}\int_{\Om}\frac{| u(x)-u(y)|^{p_2(x,y)}}{\eta^{p_2(x,y)}| x-y |^{N+s_2(x,y)p_2(x,y)}}dxdy<\infty,
        \text{ for some }\eta>0\right\}
    \end{align*}
    endowed with the norm
    $$
    \|u\|_{W}:=[u]_{s_2(\cdot),p_2(\cdot),\Omega}+\|u\|_{L^{\overline{p}_2(\cdot)}(\Om)},
    $$
    where in general $[\,\cdot\,]_{s_i(\cdot),p_i(\cdot),\mathcal{D}}$ is defined as follows
    $$
    [u]_{s_i(\cdot),p_i(\cdot),\mathcal{D}}:=\inf \left\{\eta>0:\int_{\mathcal{D}}\int_{\mathcal{D}}\frac{|u(x)-u(y)|^{p_i(x,y)}}{\eta^{p_i(x,y)}|x-y|^{N+s_i(x,y)p_i(x,y)}}\,d x d y <1 \right\},
    $$
    for $i\in\{1,2\}$ and $\mathcal{D}$ a generic set.
    Then, $(W, \|\cdot\|_{W})$ is a separable, reflexive Banach space, as shown in \cite{rs,ky-ho}.

    The following embedding result is studied in \cite{rs}. We also refer to \cite{ky-ho} where the authors proved the same result when $s(x,y)=s$ is a constant.
    \begin{lemma}[Sub-critical embedding]\label{Subcritical-embd}
        Let $\Omega$ be a smooth bounded  domain in $\mathbb{R}^N$ or $\Omega=\mathbb{R}^N$. Let $s_2(\cdot)$ and $p_2(\cdot)$ satisfy $(H_1)$-$(H_2),$  with $s_2(x,y)p_2(x,y)<N$ for any $(x,y)\in\overline{\Omega}\times\overline{\Omega}$. Assume that $\gamma(\cdot)\in C_+(\ol\Om)$ verifies $1<\gamma(x)<{p_2}_{s_2}^*(x)$ for all $x\in\ol\Om$.
        In addition, when  $\Omega=\mathbb{R}^N$, let $\gamma$ be uniformly continuous and  $\ol {p}_2(x)<\gamma(x)$ for all $x\in \mathbb{R}^N$ and $\inf_{x\in\R^N}\left({p_2}_{s_2}^*(x)-\gamma(x)\right)>0$. Then, it holds that
        \begin{equation*}
            W  \hookrightarrow
            L^{\gamma(\cdot)}(\Omega).
        \end{equation*}
        Moreover, if $\Omega$ be a smooth bounded  domain in $\mathbb{R}^N$, the embedding is compact.
    \end{lemma}

    Now we recall the following critical embedding result, given in
    \cite{zuo}, which is a consequence of \cite[Theorem 3.3]{Ho-Conc}.
    \begin{lemma}[Critical embedding]\label{crit}
        Let $\Om$ be a smooth bounded domain in $\R^N.$ Let $s_2(\cdot)$ and $p_2(\cdot)$ satisfy   $(H_1)$-$(H_2)$ and $(H_{7})$
        with $s_2(x,y)p_2(x,y)<N$ for any
        $(x,y)\in\overline{\Omega}\times\overline{\Omega}$. Let
        $\alpha(\cdot)$ satisfy $(H_{8})$. Then there exists a positive
        constant $C_\al=C_\al(N,s,p,\alpha,\Omega) $ such that
        \begin{equation*}
            \|u\|_{\alpha(\cdot)}\leq C_\al\|u\|_{W}
        \end{equation*}
        for any $u\in W$. That is, the embedding
        $W\hookrightarrow       L^{\alpha(\cdot)}(\Omega)$ is continuous.
    \end{lemma}

    \section{Functional setting}\label{sec3}

    Here we introduce the variational framework for problems \eqref{eq1} and  \eqref{eq2}. Let $i\in\{1,2\}$ and let $s_i(\cdot)$, $p_i(\cdot)$ satisfy $(H_1)$-$(H_2)$. Assume that $\beta$ verifies $(H_3)$. We set
    \begin{align*}
        |u|_{{X{_{p_i}}}}:=
        [u]_{s_i(\cdot),p_i(\cdot),\R^{2N}\setminus (\mathcal{C}\Omega)^2}
        +\|u\|_{L^{\overline{p}_i(\cdot)}(\Omega)}+\left\|\beta^{\frac{1}{\overline{p}_i(\cdot)}}u\right\|_{L^{\overline{p}_i(\cdot)}(\mathcal{C}\Omega)},
    \end{align*}
    where $\mathcal{C}\Omega=\R^N\setminus\Omega$ and
    \begin{align*}
        X^{s_i(\cdot)}_{p_i(\cdot)}:=\left\{u\colon\R^N\to \R \text{ measurable } \bigg| \  |u|_{{X}_{p_i}}<\infty\right\}.
    \end{align*}
    By following standard arguments, in \cite[Proposition 3.1]{BRW} we can see that $X^{s_i(\cdot)}_{p_i(\cdot)}$ is a reflexive Banach space with respect to the norm $|\cdot|_{X_{p_i}}$.
    Furthermore, note that the norm $|\cdot|_{X_{p_i}}$ is equivalent on $X_{p_i(\cdot)}^{s_i(\cdot)}$ to the following norm
    \begin{align*}
        \begin{split}
            \|u\|_{X_{p_i}}&=\inf\left\{\eta\geq 0 :\ \rho_{p_i}\left(\frac{u}{\eta}\right)\leq 1\right\}\\
            &=\inf\left\{\eta\geq 0: \  \iint_{\R^{2N}\setminus (\mathcal{C}\Omega)^2}\frac{|u(x)-u(y)|^{p_i(x,y)}}{\eta^{p_i(x,y)}|x-y|^{N+s_i(x,y)p_i(x,y)}}\,dx\,dy
            + \int_{\Omega}\frac{|u|^{\overline{p}_i(x)}}{\eta^{\overline{p}_i(x)}}\,dx\right.\\
            &\qquad \qquad \qquad \quad \left.
            + \int_{\mathcal{C}\Omega}\frac{\beta(x)}{\eta^{\overline{p}_i(x)}}
            |u|^{\overline{p}_i(x)}\,dx\leq 1\right\},
        \end{split}
    \end{align*}
    where  $\rho_{p_i}\colon X_{p_i(\cdot)}^{s_i(\cdot)}\to \R$ is the modular related to $X_{p_i(\cdot)}^{s_i(\cdot)}$, defined by
    \begin{align}\label{mod}
        \rho_{p_i}\left(u\right)
        &=  \iint_{\R^{2N}\setminus (\mathcal{C}\Omega)^2}\frac{|u(x)-u(y)|^{p_i(x,y)}}{|x-y|^{N+s_i(x,y)p_i(x,y)}}\,dx\,dy
        + \int_{\Omega}{|u|^{\overline{p}_i(x)}}\,dx\nonumber\\
        &\quad
        + \int_{\mathcal{C}\Omega}{\beta(x)}
        |u|^{\overline{p}_i(x)}\,dx.
    \end{align}
    However, the natural solution space to study fractional $p_1(\cdot)\&p_2(\cdot)$-Laplacian problems such as \eqref{eq1} and \eqref{eq2} is given by
    $$
    X:=X^{s_1(\cdot)}_{p_1(\cdot)}\cap X^{s_2(\cdot)}_{p_2(\cdot)}
    $$
    endowed with the norm
    $$
    |u|_X=\|u\|_{{X}_{p_1}}+\|u\|_{{X}_{p_2}}.
    $$
    Clearly $X$ is still a reflexive and separable Banach space with respect to $|\,\cdot\,|_{X}$.
    It is not difficult to see that we can make use of another norm on $X$ equivalent to $|\,\cdot\,|_{X}$, given as
    $$
    \|u\|:=\|u\|_{X}= \inf\left\{\eta\geq 0: \ \rho\left(\frac{u}{\eta}\right)\leq 1\right\},
    $$
    where the combined modular $\rho:X\to\R$ is defined as
    \begin{equation}\label{rhomod}
    \rho(u)=\rho_{p_1}(u)+\rho_{p_2}(u)
    \end{equation}
    such that $\rho_{p_1}$, $\rho_{p_2}$ are  described  as in
    \eqref{mod}.

    Arguing similarly to Proposition \ref{norm-mod}, we can get the following comparison result.
    \begin{prp}\label{norm-modular}
        Let hypotheses $(H_1)$-$(H_3)$ be satisfied, let $u \in X$ and $\eta>0$. Then:
        \begin{enumerate}
            \item[(i)]
            If $u\neq0$, then $\|u\|=\eta$ if and only if $\rho(\frac{u}{\eta})=1$;
            \item[(ii)] If
            $\|u\|<1$, then $\|u\|^{p_2^+}\leq\rho(u)\leq \|u\|^{p_1^-}$;
            \item[(iii)] If
            $\|u\|>1$, then $ \|u\|^{p_1^-}\leq \rho(u)\leq\|u\|^{p_2^+}$.
        \end{enumerate}
    \end{prp}
    \noindent
    Furthermore, it can easily be seen that
    $$\|u\|_W\leq\|u\|_{X_{p_2}}\leq\|u\|_X$$
    for any $u\in X$. From this and by Lemmas \ref{Subcritical-embd} and \ref{crit}, we get the following embedding results for the space $X$.
    \begin{lemma}\label{embd-X}
        Let $\Om$ be a smooth bounded domain in $\R^N.$ Let $s_i(\cdot)$ and $p_i(\cdot)$ satisfy $(H_1)$-$(H_2)$, for $i\in\{1,2\}$,  with $s_2(x,y)p_2(x,y)<N$ for any
        $(x,y)\in\overline{\Omega}\times\overline{\Omega}$. Also, let $\beta$ satisfy $(H_3)$.
        Assume that $\gamma(\cdot)\in C_+(\ol\Om)$ satisfies $1<\gamma(x)<{p}_{2_{s_2}}^*(x)$ for all $x\in\ol\Om$. Then,
        there exists a constant $C_\gamma=C_\gamma(s_i,p_i,N,\gamma,\Om)>0$ such that
        $$
        \|u\|_{L^{\gamma(\cdot)}(\Om)}\leq C_\gamma\|u\| \quad\mbox{for all } u\in X,
        $$
        moreover this embedding is compact.
    \end{lemma}

    \begin{lemma}\label{critical}
        Let $\Om$ be a smooth bounded domain in $\R^N$. Let $s_i(\cdot)$ and $p_i(\cdot)$ satisfy
        $(H_1)$-$(H_2)$ and $(H_4)$, for $i\in \{1,2\}$,  with $s_2(x,y)p_2(x,y)<N$ for any
        $(x,y)\in\overline{\Omega}\times\overline{\Omega}$. Also, let $\beta$ satisfy $(H_7)$. Assume that $\alpha(\cdot)\in C_+(\Om)$ satisfies $(H_8)$. Then, there exists a constant $S=S(N,s_i,p_i,\al,\Om)>0$ such that
        $$
            \|u\|_{L^{\alpha(\cdot)}(\Omega)}\leq S\|u\| \quad\mbox{for all } u\in X.
        $$
        That is, the embedding $X\hookrightarrow
        L^{\alpha(\cdot)}(\Omega)$ is continuous.
    \end{lemma}
\noindent   Denoting with $X^*$ the topological dual of $X$ and with $\langle\cdot,\cdot\rangle$ the dual pairing between $X$ and $X^*$, we can establish some properties of the combined modular function $\rho$ given in \eqref{rhomod} and its derivative.
The proof of the next result is similar to the one of \cite[Lemma 3.4]{BRW}, just noticing that, both quantities $\R^{2N}\setminus (\mathcal{C}\Omega)^2$ and $\Omega\times\Omega$ play the symmetrical role.

    \begin{lemma}\label{s+}
        Let hypotheses $(H_1)$-$(H_3)$ be satisfied. Then $\rho:X \to\mathbb R$ and $\rho':X\to X^*$ have the following properties:
        \begin{itemize}
            \item [$(i)$] The function $\rho$ is of class $C^1(X,\mathbb R)$ and $\rho':X\to X^*$ is coercive, that is, $$\frac{\langle\rho'(u), u\rangle}{\|u\|}\to\infty\text{\;\;as\;}\|u\|\to\infty.$$
            \item[$(ii)$] $\rho'$ is strictly monotone operator.
            \item[$(iii)$] $\rho'$  is a mapping of type $(S_+)$, that is, if $u_n\rightharpoonup u$ in $X$ and $\displaystyle\limsup_{n\to\infty} \langle\rho'(u_n), u_n-u\rangle\leq0$, then $u_n\to u$ strongly in $X$.
        \end{itemize}
    \end{lemma}

    We conclude this section introducing the variational setting for \eqref{eq1} and \eqref{eq2}.
    As proved in \cite[Proposition 3.6]{BRW}, the following integration by parts formula arises naturally for bounded  $C^2$-
    functions $u$ and $\varphi$ in $\mathbb{R}^{N}$
    \begin{align*}
        &\frac{1}{2}\iint_{\R^{2N}\setminus(\mathcal{C}\Omega)^2}\frac{|u(x)-u(y)|^{p_i(x,y)-2}(u(x)-u(y))(\varphi(x)-\varphi(y))}{|x-y|^{N+s_i(x,y)p_i(x,y)}}\,dx\,dy\nonumber\\
        &=\int_{\Omega}\varphi(-\Delta)^{s_i(\cdot)}_{p_i(\cdot)}u\,dx
        +\int_{\mathcal{C}\Omega}\varphi\mathcal{N}^{s_i(\cdot)}_{p_i(\cdot)}u\, dx,
    \end{align*}
    for $i\in\{1,2\}$.
    The above integration by parts formula leads to the following definitions.
    \begin{definition}
        We say that $w\in X$ is a weak solution to \eqref{eq1} if for any $\varphi\in X$, we have
        \begin{align*}
            &\mathcal{L}_\la(w,\varphi)\nonumber\\&:=\frac{1}{2}\iint_{\R^{2N}\setminus(\mathcal{C}\Omega)^2}\frac{|w(x)-w(y)|^{p_1(x,y)-2}(w(x)-w(y))(\varphi(x)-\varphi(y))}{|x-y|^{N+s_1(x,y)p_1(x,y)}}dxdy
            + \int_{\Omega}|w|^{\overline{p}_1(x)-2}w\varphi dx \nonumber\\
            &\qquad+\frac{1}{2}\iint_{\R^{2N}\setminus(\mathcal{C}\Omega)^2}\frac{|w(x)-w(y)|^{p_2(x,y)-2}(w(x)-w(y))(\varphi(x)-\varphi(y))}{|x-y|^{N+s_2(x,y)p_2(x,y)}}dxdy+
            \int_{\Omega}|w|^{\overline{p}_2(x)-2}w\varphi dx\nonumber\\
            &\qquad+\int_{\mathcal{C}\Omega} \beta(x)|w|^{\overline{p}_1(x)-2}w\varphi\,dx++\int_{\mathcal{C}\Omega} \beta(x)|w|^{\overline{p}_2(x)-2}w\varphi\,dx\nonumber\\
            &\qquad-\la\int_{\Om} b(x)|w|^{\al(x)-2}w\varphi dx- \int_{\Om}a(x)|w|^{r(x)-2}w\log|w|\varphi dx=0.
        \end{align*}
    \end{definition}
    \begin{definition}
        We say that $v\in X$ is a weak solution to \eqref{eq2} if for any $\varphi\in X$, we have
        \begin{align*}
            &\mathcal{T}_\mu(v,\varphi)\nonumber\\&:=\frac{1}{2}\iint_{\R^{2N}\setminus(\mathcal{C}\Omega)^2}\frac{|v(x)-v(y)|^{p_1(x,y)-2}(v(x)-v(y))(\varphi(x)-\varphi(y))}{|x-y|^{N+s_1(x,y)p_1(x,y)}}dxdy
            + \int_{\Omega}|v|^{\overline{p}_1(x)-2}v\varphi dx \nonumber\\
            &\qquad+\frac{1}{2}\iint_{\R^{2N}\setminus(\mathcal{C}\Omega)^2}\frac{|v(x)-v(y)|^{p_2(x,y)-2}(v(x)-v(y))(\varphi(x)-\varphi(y))}{|x-y|^{N+s_2(x,y)p_2(x,y)}}dxdy+
            \int_{\Omega}|v|^{\overline{p}_2(x)-2}v  \varphi dx\nonumber\\
            &\qquad+\int_{\mathcal{C}\Omega} \beta(x)|v|^{\overline{p}_1(x)-2}v \varphi\,dx+\int_{\mathcal{C}\Omega} \beta(x)|v|^{\overline{p}_2(x)-2}v \varphi\,dx\nonumber\\
            &\qquad-\int_{\Om} b(x)|v|^{\al(x)-2}v \varphi dx-\mu  \int_{\Om}a(x)|v|^{r(x)-2}v\log|v|\varphi dx \\
            &\qquad- \mu \int_{\Om} c(x)|v|^{\eta(x)-2}v \varphi dx=0.
        \end{align*}
    \end{definition}

    Then, the problems taken into account in the present paper have  variational structure, namely, the  solutions to the each problem can be found as critical points of the associated  energy functional.

    The energy functional associated with problem \eqref{eq1} is the functional $J_\la\colon X\to \R$ given by

    \begin{align}\label{eng1}
        J_{\la}(u) & := \frac{1}{2}\iint_{\R^{2N}\setminus(\mathcal{C}\Omega)^2}\frac{|u(x)-u(y)|^{p_1(x,y)}}{p_1(x,y)|x-y|^{N+s_1(x,y)p_1(x,y)}}dxdy+
        \int_{\Omega}\frac{|u|^{\overline{p}_1(x)}}{\overline{p
            }_1(x)}dx\nonumber\\
        &\qquad+\frac{1}{2}\iint_{\R^{2N}\setminus(\mathcal{C}\Omega)^2}\frac{|u(x)-u(y)|^{p_2(x,y)}}{p_2(x,y)|x-y|^{N+s_2(x,y)p_2(x,y)}}dxdy+
        \int_{\Omega}\frac{|u|^{\overline{p}_2(x)}}{\overline{p}_2(x)}dx\nonumber\\
        &\qquad+\int_{\mathcal{C}\Omega} \frac{\beta(x)|u|^{\overline{p}_1(x)}}{\overline{p}_1(x)}\,dx+\int_{\mathcal{C}\Omega} \frac{\beta(x)|u|^{\overline{p}_2(x)}}{\overline{p}_2(x)}\,dx\nonumber\\
        &\qquad-\la  \int_{\Omega}\frac{b(x)}{\al(x)}| u|^{\al(x)} dx- \int_{\Omega} \frac{a(x)}{r(x)} | u|^{r(x)}\log|u| dx
        + \int_{\Omega}\frac{a(x)}{r(x)^2}|u|^{r(x)} dx.
    \end{align}
    While, the energy functional associated with problem \eqref{eq2} is the functional $J_\mu\colon X\to \R$, defined by

    \begin{align}\label{eng2}
        J_{\mu}(u) & := \frac{1}{2}\iint_{\R^{2N}\setminus(\mathcal{C}\Omega)^2}\frac{|u(x)-u(y)|^{p_1(x,y)}}{p_1(x,y)|x-y|^{N+s_1(x,y)p_1(x,y)}}dxdy+
        \int_{\Omega}\frac{|u|^{\overline{p}_1(x)}}{\overline{p
            }_1(x)}dx\nonumber\\
        &\qquad+\frac{1}{2}\iint_{\R^{2N}\setminus(\mathcal{C}\Omega)^2}\frac{|u(x)-u(y)|^{p_2(x,y)}}{p_2(x,y)|x-y|^{N+s_2(x,y)p_2(x,y)}}dxdy+
        \int_{\Omega}\frac{|u|^{\overline{p}_2(x)}}{\overline{p}_2(x)}dx\nonumber\\
        &\qquad+\int_{\mathcal{C}\Omega} \frac{\beta(x)|u|^{\overline{p}_1(x)}}{\overline{p}_1(x)}\,dx+\int_{\mathcal{C}\Omega} \frac{\beta(x)|u|^{\overline{p}_2(x)}}{\overline{p}_2(x)}\,dx\nonumber\\
        &\qquad- \int_{\Omega}\frac{b(x)}{\al(x)}| u|^{\al(x)} dx- \mu \int_{\Omega} \frac{a(x)}{r(x)} | u|^{r(x)}\log|u| dx +\mu \int_{\Omega}\frac{a(x)}{r(x)^2}| u|^{r(x)} dx \\
        &\qquad-\mu \int_{\Omega}\frac{c(x)}{\eta(x)}|u|^{\eta(x)} dx.
    \end{align}
    Invoking Lemma \ref{logm} and using a direct computation from \cite[Proposition 3.8]{BRW}, we show that the functionals $J_{\la}$  and $J_{\mu}$ are well defined on $X$ and $J_\la, J_{\mu} \in C^1(X,\R)$ with
    $$
    \langle J_{\la}^{'}(w),\varphi\rangle=\mathcal{L }_\la(w,\varphi)\quad \text{for any}\quad \varphi\in X;
    $$
    $$
    \langle J_{\mu}^{'}(v),\varphi\rangle=\mathcal{T }_\mu(v,\varphi)\quad \text{for any}\quad \varphi\in X.
    $$
    Thus, the weak solutions of \eqref{eq1} and \eqref{eq2} are precisely the critical points of
    $J_{\la}$ and $J_\mu,$ respectively.

    \section{Subcritical case}\label{sec4}
    In this section, we deal with the subcritical problem \eqref{eq1}. Precisely, we
    give the proofs of Theorem \ref{mainthm} and  Theorem \ref{mainthm-2}.
    For this, we first need some technical
    lemmas to handle the logarithmic nonlinearity.
    \begin{lemma}\label{log-in}
        For every $\sigma>0$, we have:
            \begin{itemize}
            \item [$(i)$] $t^{\sigma} |\log (t)|\leq \displaystyle\frac{1}{e\sigma}$ for all $t\in (0,1];$
          \item [$(ii)$] $\log (t)\leq \displaystyle\frac{t^{\sigma}}{e\sigma}$ for all $t>1.$
            \end{itemize}
    \end{lemma}
    \begin{proof}
        $(i)$ Considering that $t\rightarrow t^{\sigma} |\log (t)| $ is a continuous function on
        $(0,1]$ with $\displaystyle \lim_{t \rightarrow 0} t^{\sigma} |\log
        (t)|=0$, which achieves the maximum at $t_0=e^{-1/\sigma}$, we can easily conclude.

        \vspace{0.05cm}
        \noindent
        $(ii)$ The second result is similar to that in \cite[Lemma 2.1]{Mingqi}.
    \end{proof}
    \begin{lemma}\label{logm}
        Let $(H_1)$-$(H_2)$ hold and let $r\in C_+(\Om)$ with $1<r^-<r^+<p_{2_{s_2}}^*(x)$ for all $x\in\Om$. Then
        \begin{align*} \int_{\Om}\frac{a(x)}{r(x)}|u|^{r(x)}
            \log |u| dx \leq C\|a\|_{L^\infty(\Om)}\max\left\{\|u\|^{r^-},\|u\|^{r^+}\right\}+\log
            \|u\|\int_{\Om}\frac{a(x)}{r(x)}|u|^{r(x)} dx, \ \ \forall u\in X\setminus{\{0\}},\end{align*} where
        $C=C(|\Om|,r,p_{2_{s_2}}^{*})>0$ is a suitable constant.
    \end{lemma}
    \begin{proof}
       Let $u\in X\setminus{\{0\}}$. Set $\Om_1 := \{x \in\Om : |u(x)| \leq \|u\|\}$ and $\Om_2 := \{x \in\Om : |u(x)| \geq \|u\|\}$. Then
        $$
        \int_\Om
        \frac{a(x)}{r(x)}|u|^{r(x)} \log \frac{|u|}{\|u\|} dx = \int_{\Om_1}
        \frac{a(x)}{r(x)}|u|^{r(x)} \log \frac{|u|}{\|u\|} dx+\int_{\Om_2}
        \frac{a(x)}{r(x)}|u|^{r(x)} \log \frac{|u|}{\|u\|} dx.
        $$
        Let us simplify the
        first integration. Then, using Lemma \ref{log-in}$(i)$ with $\sigma=r^-$, we get
        \begin{equation}
        \begin{aligned}\label{1}
            \int_{\Om_1}\frac{a(x)}{r(x)}|u|^{r(x)} \log \frac{|u|}{\|u\|} dx
            &\leq\frac{\|a\|_{L^{\infty}(\Om)}}{r^-}\max\left\{\|u\|^{r^-},\|u\|^{r^+}\right\}\int_{\Om_1}
            \left(\frac{|u|}{\|u\|}\right)^{r(x)} \left|\log \frac{|u|}{\|u\|}\right| dx\\
            &\leq \frac{\|a\|_{L^{\infty}(\Om)}|\Om|}{e(r^-)^2}\max\left\{\|u\|^{r^-},\|u\|^{r^+}\right\}.
        \end{aligned}
        \end{equation}

        Next, we calculate the second
        integral expression. For that, using Lemma \ref{log-in}$(ii)$  with
        $\sigma = (p_{2_{s_2}}^*)^--\epsilon -r^+$, for some sufficiently small $\e>0$, and using Lemma \ref{embd-X}, we obtain
        \begin{equation}
        \begin{aligned}\label{2}
            \int_{\Om_2}\frac{a(x)}{r(x)}|u|^{r(x)} \log \frac{|u|}{\|u\|} dx &\leq
            \frac{\|a\|_{L^{\infty}(\Om)}}{r^-}\int_{\Om_2} |u|^{r(x)} \log \frac{|u|}{\|u\|}
            dx\\
            &\leq \frac{\|a\|_{L^{\infty}(\Om)}}{e(p_{2_{s_2}}^{*-}-\epsilon -r^+)r^-}\int_{\Om_2} |u|^{r(x)} \left(
            \frac{|u|}{\|u\|}\right)^{p_{2_{s_2}}^{*-}-\epsilon -r^+}
            dx\\
            &\leq \frac{\|a\|_{L^{\infty}(\Om)}}{e(p_{2_{s_2}}^{*-}-\epsilon -r^+)r^-}\int_{\Om_2} |u|^{r(x)} \left(
            \frac{|u|}{\|u\|}\right)^{p_{2_{s_2}}^{*-}-\epsilon -r(x)}
            dx\\
            &\leq \frac{\|a\|_{L^{\infty}(\Om)}}{ \min(\|u\|^{p_{2_{s_{2}}}^{\ast -}-\epsilon
                    -r^{-}},\|u\|^{p_{2_{s_{2}}}^{\ast -}-\epsilon
                    -r^{+}}) e(p_{2_{s_2}}^{*-}-\epsilon -r^+)r^-}\int_{\Om_2}
            |u|^{p_{2_{s_2}}^{*-}-\epsilon }
            dx\\
            &\leq \frac{C_{p_{2_{s_2}}^{*-}-\epsilon}\|a\|_{L^{\infty}(\Om)}}{\min(\|u\|^{p_{2_{s_{2}}}^{\ast -}-\epsilon
                    -r^{-}},\|u\|^{p_{2_{s_{2}}}^{\ast -}-\epsilon
                    -r^{+}}) e(p_{2_{s_2}}^{*-}-\epsilon -r^+)r^-}\|u\|^{p_{2_{s_2}}^{*-}-\epsilon }
            \\
            &= \frac{C_{p_{2_{s_2}}^{*-}-\epsilon}\|a\|_{L^{\infty}(\Om)}}{e(p_{2_{s_2}}^{*-}-\epsilon -r^+)r^-}\min(
            \|u\|^{r^{-}},\|u\|^{r^{+}})\\
            &\leq \frac{C_{p_{2_{s_2}}^{*-}-\epsilon}\|a\|_{L^{\infty}(\Om)}}{e(p_{2_{s_2}}^{*-}-\epsilon -r^+)r^-}\max(
            \|u\|^{r^{-}},\|u\|^{r^{+}}),
        \end{aligned}
        \end{equation}
        where $C_{p_{2_{s_2}}^{*-}-\epsilon}>0$ is given by Lemma \ref{embd-X}.
        Combining \eqref{1} and \eqref{2}, we get the result.
        This ends the proof.
    \end{proof}
    \subsection{Existence of Multiple Solutions}
    We start by showing the mountain pass geometry for the functional $J_\la$, defined in \eqref{eng1}.
    \begin{lemma}\label{geo}
        Let the conditions in Theorem \ref{mainthm} hold. Then we have the following:
        \begin{enumerate}
            \item [$(i)$] There exist $\Lambda>0$, $\delta_\la>0$ and $\ell_\la>0$ such that for all $\lambda\in(0,\Lambda)$ we have
            \begin{align*}
                J_\la(u)\geq\ell_\la>0 \quad \text{for any } u\in X \text{ with }\ \|u\|=\delta_\la.
            \end{align*}
            \item [$(ii)$] There exists $\phi \in X$ with $\|\phi\|>\delta_\la$ such that $ J_\la(\phi)<0$ for all $\lambda\in(0,\Lambda)$.
            \item [$(iii)$] There exists $ \psi \in X$, with $\psi \geq0$, such that $J_\la(t \psi) <0$  for all $ t\to 0^+$ and for all $\lambda>0$.
        \end{enumerate}
    \end{lemma}
    \begin{proof}\begin{itemize}
            \item  [$(i)$] Let $u\in X$ with $\|u\|<1$. Then by definition, considering \eqref{rhomod}, we have $\|u\|_{X_{p_i}}<1$ for $i\in\{1,2\}$. Hence, using Lemmas \ref{Holder}, \ref{lemA1}, \ref{embd-X}, \ref{logm} and Proposition \ref{norm-modular}, we obtain
            \begin{equation}\label{simile}
            \begin{aligned}
                J_\la(u) & = \frac{1}{2}\iint_{\R^{2N}\setminus(\mathcal{C}\Omega)^2}\frac{|u(x)-u(y)|^{p_1(x,y)}}{p_1(x,y)|x-y|^{N+s_1(x,y)p_1(x,y)}}dxdy+\int_{\Omega}\frac{1}{\overline{p}_1(x)}|u|^{\overline{p}_1(x)}dx\\
                &\;\;\;\;\;+\frac{1}{2}\iint_{\R^{2N}\setminus(\mathcal{C}\Omega)^2}\frac{|u(x)-u(y)|^{p_2(x,y)}}{p_2(x,y)|x-y|^{N+s_2(x,y)p_2(x,y)}}dxdy+\int_{\Omega}\frac{1}{\overline{p}_2(x)}|u|^{\overline{p}_2(x)}dx\\
                &\;\;\;\;\;+\int_{\mathcal{C}\Omega} \frac{\beta(x)|u|^{\overline{p}_1(x)}}{\overline{p}_1(x)}\,dx+\int_{\mathcal{C}\Omega} \frac{\beta(x)|u|^{\overline{p}_2(x)}}{\overline{p}_2(x)}\,dx\\
                &\;\;\;\;\;-\la  \int_{\Omega}\frac{b(x)}{\al(x)}| u|^{\al(x)} dx-\int_{\Omega} \frac{a(x)}{r(x)} | u|^{r(x)}\log|u| dx + \int_{\Omega}\frac{a(x)}{r(x)^2}| u|^{r(x)} dx\\
                &\geq \frac{1}{2p_2^+}\rho(u)- \frac{2\la}{\alpha^-}\|b\|_{L^{g(\cdot)}(\Om)}\||u|^{\al(\cdot)}\|_{L^{\frac{r(\cdot)}{\al(\cdot)}}(\Om)}
                - C\|a\|_{L^\infty(\Om)}\max\left\{\|u\|^{r^-},\|u\|^{r^+}\right\}\\
                &\;\;\;\;\;-\log \|u\|\int_{\Om}\frac{a(x)}{r(x)}|u|^{r(x)}
                dx\\
                &\geq \frac{1}{2p_2^+}\|u\|^{p_2^{+}}
                -\frac{2\la}{\alpha^-}\|b\|_{L^{g(\cdot)}(\Om)}\left\{\|u\|^{\al^{-}}_{L^{r(\cdot)}(\Omega)}+\|u\|^{\al^{+}}_{L^{r(\cdot)}(\Omega)}\right\}\\&\qquad- C\|a\|_{L^\infty(\Om)}\|u\|^{ r^-}\\
                &\geq \frac{1}{2p_2^+}\|u\|^{p_2^{+}}
                -\frac{4\la}{\alpha^-}\max\{C_r^{\alpha^-},C_r^{\alpha^+}\}\|b\|_{L^{g(\cdot)}(\Om)}\|u\|^{\al^{-}}- C\|a\|_{L^\infty(\Om)}\|u\|^{ r^-}\\
                &= \frac{1}{2p_2^+}\|u\|^{p_2^{+}}-\la C_1\|u\|^{\al^{-}}- C_2\|u\|^{ r^-}\\
                &= \left(\frac{1}{2p_2^+}-\la C_1\|u\|^{\al^{-}-p_2^+}-
                C_2\|u\|^{ r^--p_2^+}\right)\|u\|^{p_2^{+}},
            \end{aligned}
            \end{equation}
            where $C_1=4\max\{C_r^{\alpha^-},C_r^{\alpha^+}\}\frac{\|b\|_{L^{g(\cdot)}(\Om)}}{\alpha^-}$ and $C_2=C\|a\|_{L^\infty(\Om)}.$
            Now for each $\la>0,$ we define the function $\tau_{\la}:(0,\infty)\to\RR$ as
            {\begin{align*}\tau_{\la}(t)=C_1\la t^{\al^- -p_2^+ }+C_2t^{r^-
                        -p_2^+  }.\end{align*}} Since   we have  $1<\al^-<p_2^+<r^-$, it
            follows that $\displaystyle \lim _{t\to0}
            \tau_{\la}(t)=\displaystyle\lim _{t\to \infty}
            \tau_{\la}(t)=\infty.$ Thus, we can find infimum of $ \tau_\la $. Note
            that by $\tau'_{\la}(t)=0,$ we obtain
            $$\tau'_{\la}(t)= {(\al^- -p_2^+) }\la C_1 +C_2 (r^--p_2^+) t^{{r}^--\al^-}=0,$$
            when $t=t^*:=\bigg( \la \frac{p_2^+-\al^- }{r^--p_2^+
            }\cdot\frac{C_1}{C_2}\bigg)^{1/(r^--\al^-)}.$ Clearly $t_*>0$. Also
            it can be checked that $\tau''_{\la}(t_*)>0$ and hence infimum of
            $\tau_{\la}(t)$ is achieved at $t_*$. Now observing that
            \begin{align}\label{New2}
                \tau_{\la}(t_*)&=\la C_1
                \bigg( \la \frac{p_2^+-\al^-}{(r^--p_2^+)}\cdot\frac{C_1}{C_2}\bigg)^{\displaystyle\frac{\al^--p_2^+}{r^--\al^-}}+C_2
                \bigg( \la \frac{p_2^+-\al^-}{(r^--p_2^+)}\cdot\frac{C_1}{C_2}\bigg)^{\displaystyle\frac{r^--p_2^+}{r^--\al^-}}\nonumber\\
                &=\la^{\displaystyle{\frac{r^--p_2^+}{r^--\al^-}}}\cdot C_3\to 0\text{ as\hspace{2mm}} \la\to 0^+,
            \end{align}
            for some constant $C_3>0$ independent of $u.$ Therefore we infer
            from \eqref{New2} that there exists $\Lambda>0 $ such that for any
            $\la\in(0,\Lambda),$ we can choose $\ell_\la>0$ and $0<\delta_\la<1$ sufficiently small such
            that
            \begin{align*}
                \text{
                    $J_\la(u)\geq\ell_\la>0 $ for all $ u\in X$ with $\| u \|=\delta_\la$.}
            \end{align*}

            \item [$(ii)$]
            Let $\lambda\in(0,\Lambda)$. Let $\xi\in C_0^\infty(\mathbb R^N)$, with $\xi>0$, $\xi=1$ in $B(x_0,R_0)$ and $\xi\geq1$ in $B^c(x_0,R_0)\cap \Omega$. Let $t>1$ sufficiently large such that $\| t\xi \|>1$. Then, using Proposition \ref{norm-modular} and $(H_6)$, we get
            \begin{align*}
                J_\la(t \xi)
                &=\frac{1}{2}\iint_{\R^{2N}\setminus(\mathcal{C}\Omega)^2}t^{p_1(x,y)} \frac{|\xi(x)-\xi(y)|^{p_1(x,y)}}{p_1(x,y)|x-y|^{N+sp_1(x,y)}}\,dx\,dy
                +\int_{\Omega}t^{\overline{p}_1(x)}\frac{|\xi|^{\overline{p}_1(x)}}{\overline{p}_1(x)}\,dx\\
                &\qquad+\frac{1}{2}\iint_{\R^{2N}\setminus(\mathcal{C}\Omega)^2}t^{p_2(x,y)} \frac{|\xi(x)-\xi(y)|^{p_2(x,y)}}{p_2(x,y)|x-y|^{N+sp_2(x,y)}}\,dx\,dy
                +\int_{\Omega}t^{\overline{p}_2(x)}\frac{|\xi|^{\overline{p}_2(x)}}{\overline{p}_2(x)}\,dx\\
                &\qquad+\int_{\mathcal{C}\Omega}t^{\overline{p}_1(x)} \frac{\beta(x)|\xi|^{\overline{p}_1(x)}}{\overline{p}_1(x)}\,dx
                +\int_{\mathcal{C}\Omega}t^{\overline{p}_2(x)} \frac{\beta(x)|\xi|^{\overline{p}_2(x)}}{\overline{p}_2(x)}\,dx \\
                &\qquad-\la  \int_{\Omega}\frac{b(x)}{\al(x)}| t\xi|^{\al(x)} dx- \int_{\Omega} \frac{a(x)}{r(x)} | t\xi|^{r(x)}\log|t\xi| dx +  \int_{\Omega}\frac{a(x)}{r(x)^2}| t\xi|^{r(x)} dx\\
                &\leq t^{p_2^+} \rho(\xi)-\lambda t^{\alpha^{-}}\frac{b^-}{\alpha^+} \int_{\Omega}  |
                \xi|^{\alpha(x)}dx\\
                &\qquad+\frac{\|a\|_{L^\infty(\Om)}t^{r^{+}}}{(r^{-})^{2}}\int_{\Omega}|\xi|^{r(x)}\,dx-\frac{1}{r^+} \int_{B^{c}(x_0,R_0)\cap \Omega}  a(x)| t\xi|^{r(x)}\log|t\xi| dx\\
                &= t^{p_2^+} \rho(\xi)-\lambda t^{\alpha^{-}}\frac{b^-}{\alpha^+} \int_{\Omega}  |
                \xi|^{\alpha(x)}dx
                +\frac{\|a\|_{L^\infty(\Om)}}{({r^{-}})^{2}}t^{r^{+}}\int_{\Omega}|\xi|^{r(x)}\,dx-\frac{t^{r^{+}}}{r^+} \int_{B^{c}(x_0,R_0)\cap \Omega}  a(x) |\xi|^{r^{+}}\log|t\xi| dx\\
                &= t^{r^{+}} \left[t^{p_2^{+}-r^{+}} \rho(\xi)
                +\frac{\|a\|_{L^\infty(\Om)}}{({r^{-}})^{2}}\int_{\Omega}|\xi|^{r(x)}\,dx-\frac{1}{r^+} \int_{B^{c}(x_0,R_0)\cap \Omega}  a(x) |\xi|^{r^{+}}\log|\xi| dx\right]\\
                &\qquad-t^{r^{+}} \left[\frac{1}{r^+} \int_{B^{c}(x_0,R_0)\cap \Omega}  a(x) |\xi|^{r^{+}}\log|t|dx-\lambda t^{\alpha^{-}-r^{+}}\frac{b^-}{\alpha^+} \int_{\Omega}  |
                \xi|^{\alpha(x)}dx\right].
            \end{align*}
            Since $\alpha^{-}<p_2^{+}<r^{+}$ by $(H_6)$, from above  we have $J_\la(t\xi)\to-\infty$ as $t\to\infty$. Hence, we can find $\widetilde t>0,$ sufficiently large such that $\phi:=\widetilde t\xi$ verifies $\|\phi\|\geq\delta_\la$ and $J_\la(\phi)<0$.

            \vspace{0.1cm}
            \item [$(iii)$]  Let $\lambda>0$. Let $\psi\in C_0^\infty(\mathbb R^N)$, with $\psi=0$ in $B(x_0,R_0)$ and $\psi\geq 0$ in $B^c(x_0,R_0)\cap \Omega$. Let
            $t>0$ sufficiently small such that $\| t\psi \|<1$. Then, by  Proposition
            \ref{norm-modular}, Lemma \ref{logm} and $(H_6)$, we obtain \vspace{-.2cm}
            \begin{align*}
                J_\la(t\psi) \leq& t^{p_1^-} \rho(\psi)
                    +\frac{\|a\|_{L^\infty(\Om)}t^{r^{-}}}{{r^{-}}^{2}}\int_{\Omega}|\psi|^{r(x)}\,dx-\frac{1}{r^+}
                    \int_{B^{c}(x_0,R_0)\cap \Omega}  a(x)|
                    t\psi|^{r^{+}}\log|t\psi| dx\\
                    &-\lambda
                    t^{\alpha^{+}}\frac{b^-}{\alpha^+} \int_{\Omega}  |
                    \psi|^{\alpha(x)}dx,
            \end{align*}
            {which shows that
                \begin{align*}
                    \frac{J_\la(t\psi)}{
                        t^{\alpha^{+}}}  &\leq t^{p_1^{-}-\alpha^{+}} \rho(\psi)
                    +\frac{a^{+}t^{r^{-}-\alpha^{+}}}{{r^{-}}^{2}}\int_{\Omega}|\psi|^{r(x)}\,dx\\
                    &\;\;\;-\frac{1}{r^+} \int_{B^{c}(x_0,R_0)\cap\Omega}  a(x)|
                    t\psi|^{r^{+}-\alpha^{+}}\log|t\psi| dx-\lambda
                    \frac{b^-}{\alpha^+} \int_{\Omega}  | \psi|^{\alpha(x)}dx.
                \end{align*}
                Thus $$\displaystyle \lim_{t\rightarrow
                    0}\frac{J_\la(t\psi)}{t^{\alpha^{+}}}\leq -\lambda
                \frac{b^-}{\alpha^+} \int_{\Omega}  | \psi(x)|^{\alpha(x)}dx <0,$$
                since  $\al^+<p_1^-<r^{+}$ by $(H_6)$. \\
                This completes the proof.}
        \end{itemize}
    \end{proof}

Let $k\in\mathbb R$. We recall that a functional $J_\la:X\to\mathbb R$ fulfills the Palais-Smale condition $(PS)_k$ if any sequence $(u_n)_{n\geq1}\subset X$ satisfying
\begin{equation}\label{3.7}
J_\la(u_n)\to k\quad\mbox{ and }\quad\|J_\la'(u_n)\|_{X^*}\to 0\quad\mbox{ as }n\rightarrow\infty,
\end{equation}
admits a convergent subsequence in $X$. Now we study the Palais-Smale condition for the functional $J_\la$.
    \begin{lemma}\label{bounded}
        Let the hypotheses in Theorem \ref{mainthm} hold. Then,  the functional $J_\la$ satisfies the $(PS)_k$-condition for any $k\in\R$ and any $\lambda>0$.
    \end{lemma}

    \begin{proof}
        Let $k\in\R$ and $\lambda>0$.
        Let  $(u_n)_{n\geq1}\subset X$ be a $(PS)_k$ sequence of the
        functional $J_\la$, that is satisfying \eqref{3.7}. We first prove that $(u_n)_{n\geq1}$ is bounded in $X$. Indeed, if  $(u_n)_{n\geq1}$ is
        unbounded in $X$,  up to a subsequence, we may assume that $\|u_n\|>1$ for $n\in\mathbb N$ large enough.
        Hence, by \eqref{3.7} there exists some constant $ {d}>0$ such that, using Lemmas   \ref{Holder}, \ref{lemA1}, \ref{embd-X}, \ref{log-in} and Proposition \ref{norm-modular}, we deduce as $n\to\infty$
            \begin{align}\label{ps}
                & o_n(1)+d+k\|u_n\|\nonumber\\
                &\geq J_\la(u_n) -\frac{1}{r^-} \langle J_\la'(u_n),u_n\rangle\nonumber\\
                &\geq\left(\frac{1}{2p_2^+}-\frac{1}{r^-}\right)\rho(u_n)-\la\left(\frac{1}{\al^-}-\frac{1}{r^-}\right)\int_\Om b(x)|u_n|^{\al(x)}dx\nonumber\\ &\qquad+\int_{\Omega}\left(\frac{1}{r^-}
                -\frac{1}{r(x)}\right) a(x) | u_n|^{r(x)}\log|u_n| dx + \int_{\Omega}\frac{a(x)}{r(x)^2}| u_n|^{r(x)} dx\nonumber\\
                &\geq\left(\frac{1}{2p_2^+}-\frac{1}{r^-}\right)\rho(u_n)-\la\left(\frac{1}{\al^-}-\frac{1}{r^-}\right)\int_\Om b(x)|u_n|^{\al(x)}dx\nonumber\\ &\qquad-{
                    \left(\frac{1}{r^-}-\frac{1}{r^+}\right)\int_{ \Omega\cap \{|u_n(x)|<1\}} a(x) | u_n|^{r(x)}|\log|u_n|\,| dx}
                    + \int_{\Omega}\frac{a(x)}{r(x)^2}| u_n|^{r(x)}
                dx\nonumber\\
                &\geq\left(\frac{1}{2p_2^+}-\frac{1}{r^-}\right)\|u_n\|^{p_1^-}-\la\left(\frac{1}{\al^-}-\frac{1}{r^-}\right)\|b\|_{L^{g(\cdot)}(\Om)}
                \left\{\|u_n\|^{\al^{-}}_{L^{r(\cdot)}(\Omega)}+\|u_n\|^{\al^{+}}_{L^{r(\cdot)}(\Omega)}\right\}\nonumber\\
                &\qquad-\left(\frac{1}{r^-}-\frac{1}{r^+}\right)\frac{|\Omega|\|a\|_{L^\infty(\Om)}}{er^-}\nonumber\\
                &\geq\left(\frac{1}{2p_2^+}-\frac{1}{r^-}\right)\|u_n\|^{p_1^-}-2\la\left(\frac{1}{\al^-}-\frac{1}{r^-}\right)\max\{C_r^{\alpha^-},C_r^{\alpha^+}\}\|b\|_{L^{g(\cdot)}(\Om)}\|u_n\|^{\al^{+}}\nonumber\\
                &\qquad-\left(\frac{1}{r^-}-\frac{1}{r^+}\right)\frac{|\Omega|\|a\|_{L^\infty(\Om)}}{er^-}.
                \end{align}
        Since $1<\al^+<p_1^-<2p_2^+<r^-$ by $(H_6)$, when $\|u_n\|\to\infty$,
        from the above expression, we get a contradiction and hence the
        sequence $(u_n)_{n\geq1} $ is bounded in $X.$

        Since $X$ is reflexive and using Lemma \ref{embd-X}, there exists a subsequence, still denoted by $(u_n)_{n\geq1}$, and $u\in X$ such that
        \begin{equation}\label{convergenze}
        \begin{aligned}
            u_n \rightharpoonup u\;\;\text{ in}\; X,\qquad u_n\to u\;\;\text { in} \;L^{\gamma(\cdot)}(\Om),\qquad u_n \to u\;\; \text{a.e. in}\; \Om,
        \end{aligned}
        \end{equation}
        as $n\to\infty$, for any $\gamma\in C_+(\overline{\Omega})$ satisfying $1 < \gamma(x) <{p_2}_{s_2}^*(x)$ for any $x\in\overline{\Omega}$.
        Now we  show that
        \begin{equation}\label{ll}
            \displaystyle \lim_{n\rightarrow \infty}
            \langle\rho'(u_n)-\rho'(u),u_n-u\rangle=0.
        \end{equation}
        Let $\sigma\in\left(0,(p_{2_{s_2}}^*)^- - r^{+}\right)$. For any
            measurable subset $\Om'\subset \Om$, by Lemmas \ref{embd-X} and \ref{log-in} we have
            \begin{align*}
                \int_{\Om'} a(x) |u_n|^{r(x)}|\log|u_n|| dx &=\displaystyle
                \int_{\Om'\cap \{|u_n|\leq 1\}}
                a(x) |u_n|^{r(x)}|\log |u_n|| dx
                +\displaystyle \int_{\Om'\cap \{|u_n|> 1\}}
                a(x) |u_n|^{r(x)}|\log|u_n|| dx \\
                &\leq\frac{|\Omega'|\|a\|_{L^\infty(\Om)}}{er^-} +
                \frac{\|a\|_{L^\infty(\Om)}}{e\sigma}\int_{\Om'} a(x) |u_n|^{r^{+}+\sigma} dx\\
                &\leq\frac{|\Omega'|\|a\|_{L^\infty(\Om)}}{er^-}+M\frac{|\Omega'|\|a\|_{L^\infty(\Om)}C_{r^{+}+\sigma}^{r^{+}+\sigma}}{e\sigma},
            \end{align*}
            with $M=\sup\|u_n\|^{r^{+}+\sigma}<\infty$,
            which implies that the sequence $(a(x)
                |u_n|^{r(x)}|\log|u_n(x)||)_{n\geq1}$ is uniformly bounded and
            equi-integrable in $L^{1}(\Omega)$. Moreover, by \eqref{convergenze} it easy to see that as
            $n\to\infty$, we get
            $$a(x) |u_n(x)|^{r(x)}|\log|u_n(x)|\,|\rightarrow a(x)
            |u(x)|^{r(x)}|\log|u(x)|\,|\;\; \text{a.e in $\Omega$}.$$
            Thus, the   Vitali's convergence theorem yields that
        \begin{equation}\label{ll1}
            \displaystyle \lim_{n \rightarrow \infty}\int_{\Om}
            a(x) |u_n|^{r(x)}\log|u_n| dx=\int_{\Om}
            a(x) |u|^{r(x)}\log |u| dx.
        \end{equation}
        Similarly, we can prove that
        \begin{equation}\label{ll2}
            \lim_{n \rightarrow \infty}\int_\Om  a(x) u | u_n|^{r(x)-2}u_n\log|u_n| dx= \int_\Om  a(x) | u|^{r(x)}\log|u|
            dx
        \end{equation}
        and
        \begin{equation}\label{ll3}
            \lim_{n \rightarrow \infty}\int_\Om  a(x) u_n | u|^{r(x)-2}u \log|u| dx= \int_\Om  a(x) | u|^{r(x)}\log|u|
            dx.
        \end{equation}
        Consequently, from \eqref{ll1}, \eqref{ll2} and \eqref{ll3}, we conclude that
        \begin{align}\label{ll4}
            \lim_{n \rightarrow \infty}\int_\Om a(x)|
            u_n|^{r(x)-2}u_n(u_n-u)\log|u_n|dx=\lim_{n \rightarrow
                \infty}\int_\Om a(x)| u|^{r(x)-2}u(u_n-u)\log|u|dx.
        \end{align}
        While, \eqref{convergenze} and $(H_6)$ we can deduce that
        \begin{equation}\label{ll5}
            \displaystyle \lim_{n \rightarrow \infty}\int_{\Om} b(x)|u_n|^{\alpha (x)}dx=\int_{\Om} b(x)|u|^{\alpha (x)}dx.
        \end{equation}
        Combining \eqref{3.7}, \eqref{ll4} and \eqref{ll5}, we prove claim \eqref{ll}. Hence, using Lemma
        \ref{s+}, we get our desired result.
    \end{proof}

    \noindent{\bf Proof of Theorem \ref{mainthm}:}
    Let $\lambda\in(0,\Lambda)$, with $\Lambda>0$ given in Lemma \ref{geo}$(i)$.
    Note that   { $J_{\la}(0)=0.$ } From Lemma \ref{geo}$(i),(ii)$, it
    follows that $J_\la$ admits a mountain pass geometry for all $\la\in(0,\Lambda)$. Also,
    by Lemma \ref{bounded},
    the functional $J_\la$ satisfies the Palais-Smale condition $(PS)_k$
    for any $k\in\RR$. Hence, by applying mountain pass theorem, we
    infer that for $\la\in(0,\Lambda),$ there exists a nontrivial weak
    solution $u_{1,\la}\in X$ of  \eqref{eq1} with
    \begin{align}\label{0.123}
    J_\la(u_{1,\la})\geq\ell_\la>0,
    \end{align}
    with $\ell_\la>0$ given by Lemma \ref{geo}$(i)$.
    \par Next, we prove the existence of the second weak solution of \eqref{eq1}. From Lemma \ref{geo}$(iii),$ it yields that
    \begin{align}\label{6.0}
        \displaystyle\inf _{u\in \overline {B}_{\delta_\la}(0)}
        J_\la(u)=m_\la<0,
    \end{align}
    where $\overline{B}_{\delta_\la}(0)=\{u\in X :
    \|u\|\leq\delta_\la\}$.
    From this, by also Lemma \ref{geo}$(i)$ and the Ekeland's variational
    principle in $\overline{B}_{\delta_\la}(0)$,  there exists a sequence $(v_n)_{n\geq1}\subset B_{\delta_\la} (0)$ such that
    \begin{align}\label{0.6.1}
        m_\la\leq J_\la(v_n)\leq m_\la+\frac{1}{n}
        \end{align}\quad\text{and}\quad\begin{align}\label{cc}
        J_\la(v_n)<J_\la(u)+\frac{1}{n}\|u-v_n\|,
    \end{align}
for all $n\in\mathbb N$ and all $u\in \overline{B}_{\delta_\la}(0)$. Then \eqref{0.6.1} implies that $J_\la\to m_\la$.
Fix $n\in\mathbb{N}$, for all $w\in\partial B_1(0)$, where
$
\partial B_1(0)=\{u\in X:\,\, \|u\|=1\},
$
and for all $\varepsilon>0$ so small that $v_n+\varepsilon\,w\in\overline{B}_{\delta_\la}(0)$, using \eqref{cc}, we have
$$
J_\la(v_n+\varepsilon\,w)-J_\la(v_n)\ge-\frac{\varepsilon}{n}
$$
by \eqref{0.6.1}. Since $J_\la$ is G\^ateaux differentiable in $X$, we get
\begin{equation*}
\lim_{\varepsilon\rightarrow 0}
\frac{J_\la(v_n+\varepsilon\,w)-J_\la (v_n)}{\varepsilon}\ge-\frac{1}{n}
\end{equation*}
for all $w\in\partial B_1(0)$. Now replacing $w$ by $-w$ in the above argument, we obtain \begin{equation*}
    \lim_{\varepsilon\rightarrow 0}
    \frac{J_\la(v_n)-J_\la (v_n-\e w)}{\varepsilon}\le \frac{1}{n}.
\end{equation*} Hence
\begin{equation*}
\big|\langle J_\la'(w_n),w\rangle\big|\le\frac{1}{n},
\end{equation*}
since $w\in\partial B_1(0)$ is arbitrary. Consequently, the sequence $(v_n)_{n\ge1}\subset B_{\delta_\la}(0)$ verifies \eqref{3.7} with $k=m_\la$.
    Therefore, from
    Lemma \ref{bounded} and \eqref{6.0}, it yields that there exists
    $u_{2,\la}\in \overline{B}_{\delta_\la}(0)\subset X$ such that
$v_n\to u_{2,\la}$ strongly in $X$ as $n\to\infty$ with
    \begin{align}\label{6.7}
        J_\la(u_{2,\la})=m_\la<0 .
    \end{align}
    Thus, we get that $u_{2\la} $ is a nontrivial weak solution of
    \eqref{eq1}. Now from \eqref{0.123} and \eqref{6.7}, we have
    $J_\la(u_{1,\la})>0>J_\la(u_{2,\la})$, and hence $ u_{1,\la} \not=
    u_{2,\la}.$ This completes the proof of the theorem.

    \subsection{Existence of Ground state Solution}

    To look for the ground state solution of \eqref{eq1}, we consider the following
    minimizing problem
    $$m^{\ast}=\inf_{u\in X}J_{\lambda}(u).$$
        Then, we are able to study the basic properties for $J_\la$.
\begin{lemma}\label{ground}
Let the hypotheses in Theorem \ref{mainthm-2} hold.
Then:
\begin{enumerate}
\item[$(i)$] For every $\lambda>0$,  $J_{\lambda}$ is coercive.
\item[$(ii)$] The functional $J_\la$ satisfies the $(PS)_k$-condition for
any $k\in\R$ and any $\lambda>0$.
\end{enumerate}
\end{lemma}
\begin{proof}
        \begin{itemize}
            \item [$(i)$] Let $\la>0$ and let $u\in X$ with $\|u\|>1$. In light of Lemmas \ref{Holder}, \ref{lemA1}, \ref{embd-X}, \ref{log-in} and Proposition \ref{norm-modular}, we infer that
 \begin{align*}
        J_{\la}(u) & := \frac{1}{2}\iint_{\R^{2N}\setminus(\mathcal{C}\Omega)^2}\frac{|u(x)-u(y)|^{p_1(x,y)}}{p_1(x,y)|x-y|^{N+s_1(x,y)p_1(x,y)}}dxdy+
        \int_{\Omega}\frac{|u|^{\overline{p}_1(x)}}{\overline{p
            }_1(x)}dx\nonumber\\
        &\qquad+\frac{1}{2}\iint_{\R^{2N}\setminus(\mathcal{C}\Omega)^2}\frac{|u(x)-u(y)|^{p_2(x,y)}}{p_2(x,y)|x-y|^{N+s_2(x,y)p_2(x,y)}}dxdy+
        \int_{\Omega}\frac{|u|^{\overline{p}_2(x)}}{\overline{p}_2(x)}dx\nonumber\\
        &\qquad+\int_{\mathcal{C}\Omega} \frac{\beta(x)|u|^{\overline{p}_1(x)}}{\overline{p}_1(x)}\,dx+\int_{\mathcal{C}\Omega} \frac{\beta(x)|u|^{\overline{p}_2(x)}}{\overline{p}_2(x)}\,dx\nonumber\\
        &\qquad-\la  \int_{\Omega}\frac{b(x)}{\al(x)}| u|^{\al(x)} dx- \int_{\Omega} \frac{a(x)}{r(x)} | u|^{r(x)}\log|u| dx
        + \int_{\Omega}\frac{a(x)}{r(x)^2}|u|^{r(x)} dx\nonumber\\
        &\geq
        \frac{1}{2p_{2}^{+}}\rho(u)-\frac{2\la}{\alpha^-}\|b\|_{L^{g(\cdot)}(\Om)}\||u|^{\al(\cdot)}\|_{L^{\frac{r(\cdot)}{\al(\cdot)}}(\Om)}-\int_{ \{x\in \Om, \ \
        |u(x)|<1\}}a(x)|u|^{r(x)}\log|u|dx\\
        &\geq \frac{1}{2p_{2}^{+}}
        \|u\|^{p_{1}^{-}}-\frac{4\la}{\alpha^-}\max\{C_r^{\alpha^-},C_r^{\alpha^+}\}\|b\|_{L^{g(\cdot)}(\Om)}\|u\|^{\al^{+}}+\frac{\int_{\Om}a(x)dx}{er^{-}}.
    \end{align*}
    Thus $J_{\la}$ is coercive since $\alpha^{+}<p_{1}^{-}$ by $(H_6)$.\\
    \item [$(ii)$] Let $k\in\R$ and $\lambda>0$.
        Let  $(u_n)_{n\geq1}\subset X$ be a $(PS)_k$ sequence of the
        functional $J_\la$, that is satisfying \eqref{3.7}. From
        $(i)$, we can deduce that $(u_n)_{n\geq1}$ is bounded in $X$. The
        rest of the proof is similar to that of Lemma \ref{bounded}.
\end{itemize}
\end{proof}
\noindent
    \textbf{Proof of Theorem \ref{mainthm-2}:}
    Let $\lambda>0$. From Lemma \ref{ground} and Ekeland's variational principle, we infer that there
    exists $\widetilde u_{\la}\in X$ such that  $J_\la(\widetilde u_\la)=m^*$ and
    $J'_\la(\widetilde u_\la)=0$. Therefore, $\widetilde u_\la$ is a
    ground state solution of \eqref{eq1}. It remains to show that $\widetilde u_\la\neq
    0$. For this, by the same argument used in the proof of Lemma
    \ref{geo}(iii), we can find $w\in X$ such that $J_{\la}(w)<0$.
    Thus $m^{\ast}\leq J_{\la}(w)<0$. Hence, we conclude that $\widetilde
    u_\la\neq
    0$.\\
    This ends the proof.
    \section{Critical case}\label{sec5}
    In this section, we give the proof of Theorem \ref{mainthm-3}. First we prove the mountain pass structure for the functional $J_\mu$, defined in \eqref{eng2}.
    \begin{lemma}\label{geo2}
        Let the hypotheses in Theorem \ref{mainthm-3} hold. Then, we have the following:
        \begin{enumerate}
            \item [$(i)$] For any $\mu>0$ there exist $\ell_\mu$, $\delta_\mu>0$ such that
            \begin{align*}
                J_\mu(u)\geq \ell_\mu>0 \quad \text{for any } u\in X \text{ with }\ \|u\|=\delta_\mu.
            \end{align*}
            \item [$(ii)$] For any $\mu>0$ there exists $w \in X$ such that $\|w\|>\delta_\mu$ and $ J_\mu(w)<0$.
        \end{enumerate}
    \end{lemma}
    \begin{proof}
        \begin{itemize}
            \item [$(i)$]
            Let $u\in X$ with $\|u\|<\min\{1,1/C_\eta,1/S\}$, with $C_\eta>0$ given in Lemma \ref{embd-X}. Then by definition, considering \eqref{rhomod}, we have $\|u\|_{X_{p_i}}<1$ for $i\in\{1,2\}$. Thus, arguing similarly to \eqref{simile} and using here Lemma
            \ref{critical}, we obtain
            \begin{align*}
                J_\mu(u) & = \frac{1}{2}\iint_{\R^{2N}\setminus(\mathcal{C}\Omega)^2}\frac{|u(x)-u(y)|^{p_1(x,y)}}{p_1(x,y)|x-y|^{N+s_1(x,y)p_1(x,y)}}dxdy+\int_{\Omega}\frac{1}{\overline{p}_1(x)}|u|^{\overline{p}_1(x)}dx\\
                &\;\;\;\;\;+\frac{1}{2}\iint_{\R^{2N}\setminus(\mathcal{C}\Omega)^2}\frac{|u(x)-u(y)|^{p_2(x,y)}}{p_2(x,y)|x-y|^{N+s_2(x,y)p_2(x,y)}}dxdy+\int_{\Omega}\frac{1}{\overline{p}_2(x)}|u|^{\overline{p}_2(x)}dx\\
                &\;\;\;\;\;+\int_{\mathcal{C}\Omega} \frac{\beta(x)|u|^{\overline{p}_1(x)}}{\overline{p}_1(x)}\,dx+\int_{\mathcal{C}\Omega} \frac{\beta(x)|u|^{\overline{p}_2(x)}}{\overline{p}_2(x)}\,dx\\
                &\;\;\;\;\;- \int_{\Omega}\frac{b(x)}{\al(x)}| u|^{\al(x)} dx-\mu\int_{\Omega} \frac{a(x)}{r(x)} | u|^{r(x)}\log|u| dx
                +\mu \int_{\Omega}\frac{a(x)}{r(x)^2}| u|^{r(x)} dx
                -\mu \int_{\Omega}\frac{c(x)}{\eta(x)}| u|^{\eta(x)} dx\\
                &\geq \frac{1}{2p_2^+}\rho(u)- \frac{\|b\|_{L^{\infty}(\Om)}}{\alpha^-}\|u\|_{L^{\al(\cdot)}(\Om)}^{\al^-}
                - \mu C\|a\|_{L^\infty(\Om)}\max\left\{\|u\|^{r^-},\|u\|^{r^+}\right\}\\
                &\;\;\;\;\;-\mu \|a\|_{L^\infty(\Om)}\log \|u\|\int_{\Om}|u|^{r(x)}
                dx-\mu\frac{\|c\|_{L^\infty(\Om)}}{\eta^-}\|u\|_{L^{\eta(\cdot)}(\Om)}^{\eta^-}\\
                &\geq \frac{1}{2p_{2}^{+}}\|u\|^{p_2^{+}}
                - \widetilde{C}_1\|u\|^{\al^{-}}- \mu  \widetilde{C}_2\|u\|^{ r^-}-\mu \widetilde{C}_3\|u\|^{ \eta^-}\\
                &\geq \frac{1}{2p_{2}^{+}}\|u\|^{p_{2}^{+}}-( \widetilde{C}_1+ \mu\widetilde{ C}_2+\mu\widetilde{ C}_3)\|u\|^{\eta^-},
            \end{align*}
            since $\eta^-<r^-<\al^-$ by $(H_{10})$,
            where $\widetilde{C}_1=S^{\alpha^-}\|b\|_{L^{\infty}(\Om)}/\alpha^-$, $\widetilde{C}_2=C\|a\|_{L^\infty(\Om)}$ and
            $\widetilde{C}_3=\|c\|_{L^\infty(\Om)}C_\eta^{\eta^-}/\eta^-$.
            Therefore,  we infer that
            \begin{align*}
                \text{
                    $J_\mu(u)\geq \ell_\mu>0 $ for all $ u\in X$ with $\| u \|=\delta_\mu$,}
            \end{align*}
        where $$\ell_\mu=\frac{1}{4p_{2}^{+}} \ \ \mbox{and} \ \
            \delta_\mu<\min\left\{1,\left[\frac{1}{4p_{2}^{+}( \ol C_{1}+\mu \ol C_{2}+\mu
                \ol C_3)}\right]^{\frac{1}{\eta^--p_{2}^{+}}}\right\}.$$
            \item [$(ii)$] Choose a positive
            function $\xi\in X$ with $\|\xi\|=1$. Then for large $t>1$, we have
            \begin{align*}
                J_\mu(t \xi)
                &=\frac{1}{2}\iint_{\R^{2N}\setminus(\mathcal{C}\Omega)^2}t^{p_1(x,y)} \frac{|\xi(x)-\xi(y)|^{p_1(x,y)}}{p_1(x,y)|x-y|^{N+sp_1(x,y)}}\,dx\,        dx+\int_{\Omega}t^{\overline{p}_1(x)}\frac{|\xi|^{\overline{p}_1(x)}}{\overline{p}_1(x)}\,dx\\
                &\qquad+\frac{1}{2}\iint_{\R^{2N}\setminus(\mathcal{C}\Omega)^2}t^{p_2(x,y)} \frac{|\xi(x)-\xi(y)|^{p_2(x,y)}}{p_2(x,y)|x-y|^{N+sp_2(x,y)}}\,dx\,dy
                +\int_{\Omega}t^{\overline{p}_2(x)}\frac{|\xi|^{\overline{p}_2(x)}}{\overline{p}_2(x)}\,dx\\
                &\qquad +\int_{\mathcal{C}\Omega}t^{\overline{p}_1(x)} \frac{\beta(x)|\xi|^{\overline{p}_1(x)}}{\overline{p}_1(x)}\,dx
                +\int_{\mathcal{C}\Omega}t^{\overline{p}_2(x)} \frac{\beta(x)|\xi|^{\overline{p}_2(x)}}{\overline{p}_2(x)}\,dx\\
                &\qquad-  \int_{\Omega}\frac{b(x)}{\al(x)}| t\xi|^{\al(x)} dx- \mu \int_{\Omega} \frac{a(x)}{r(x)} | t\xi(x)|^{r(x)}\log|t\xi(x)| dx +\mu  \int_{\Omega}\frac{a(x)}{r(x)^2}| t\xi(x)|^{r(x)} dx\\
                &\qquad-\mu \int_{\Omega} \frac{c(x)}{\eta(x)}|t\xi|^{\eta (x)}dx\\
                &\leq t^{p_2^+} \rho(\xi)
                +\mu\frac{a^+t^{r^{+}}}{(r^{-})^{2}}\int_{\Omega}|\xi|^{r(x)}\,dx+\mu t^{r^+}\int_{\Omega} \frac{a(x)}{r(x)} |\xi(x)|^{r(x)}|\log|t\xi(x)|\,|dx
                - \frac{t^{\alpha^{-}}}{\alpha^+} \int_{\Omega} b(x) |
                \xi(x)|^{\alpha(x)}dx\\
                &\qquad-\mu t^{\eta^-}\int_{\Omega} \frac{c(x)}{\eta(x)}|\xi|^{\eta (x)}dx.
            \end{align*}
            Since $p^{+}_{2}<\eta^-<r^{+}<\alpha^{-}$, passing to the limit as
            $t\rightarrow \infty$ we get $J_{\mu}(t\xi)<0$. Thus, the assertion
            follows by taking $w=t^*\xi$, with $t^*$ sufficiently large.
        \end{itemize}
    \end{proof}
Let us fix $\mu>0$
    and let us define the critical mountain pass level as
    \begin{equation}\label{cmu}
        c_{\mu}=\inf\limits_{\nu\in\Gamma}\max\limits_{t\in[0,1]}J_{\mu}(\nu(t)),
    \end{equation}
    where
    \begin{equation*}
        \Gamma=\{\nu\in
        C^{1}([0,1],X):\,\,\nu(0)=0,\,\,J_{\mu}(\nu(1))<0\}.
    \end{equation*}
    Evidently, $c_{\mu}>0$ according to Lemma \ref{geo2}. Thus, by mountain pass theorem in \cite[Theorem 1.15]{willem}, there exists a Palais-Smale sequence $(v_n)_{n\geq1}$ such that
    \begin{align}\label{pss}
    J_\mu(u_n)\to c_\mu\quad\mbox{ and }\quad\|J_\mu'(v_n)\|_{X^*}\to 0\quad\mbox{ as }n\rightarrow\infty.
    \end{align}
    Before verifying the validity of compactness of Palais-Smale
    sequences, we need to establish an asymptotic condition for the
    level $c_{\mu}$. This result will be crucial to overcome the lack of
    compactness due to the presence of a critical nonlinearity in
    \eqref{eq2}.
    \begin{lemma}\label{mu}
        Let the hypotheses in Theorem \ref{mainthm-3} hold.
        Then
        $$\displaystyle \lim_{\mu \rightarrow
            \infty}c_{\mu}=0.$$
    \end{lemma}
    \begin{proof}
        Fix $\mu>0$ and take $w\in X$ as given in Lemma \ref{geo2} $(ii)$.
        In light of Lemma \ref{geo2}, there exists $t_{\mu}>0$ such that
        $J_{\mu}(t_{\mu}w)=\max\limits_{t\geq0}J_{\mu}(tw)$. Therefore,
        $\langle J^{'}_{\mu}(t_{\mu}w),w\rangle=0$ and so, from Lemma
        \ref{log-in} with $\sigma=r^--\eta^+$ and condition $(H_{10})$, we obtain
        \begin{equation}
        \begin{aligned}\label{mueq1}
            \rho(t_{\mu}w)&= \displaystyle
            \int_{\Omega}b(x)t_{\mu}^{\alpha(x)}|\omega|^{\alpha(x)}dx+ \mu
            \displaystyle \int_{\Omega}
            a(x)|t_{\mu}w|^{r(x)} \log |t_{\mu} w|dx+\mu \int_{\Omega} c(x)|t_\mu w|^{\eta(x)}dx \\
            &\geq \displaystyle
            \int_{\Omega}b(x)t_{\mu}^{\alpha(x)}|w|^{\alpha(x)}dx+
            \mu \displaystyle \int_{\Omega\cap\{ |t_{\mu}w(x)|>1\}}
            a(x)|t_{\mu} w|^{r(x)} \log |t_{\mu}w|)dx\\
            &\qquad+\mu \displaystyle\int_{ \Omega\cap\{ |t_{\mu}w(x)|<1\}}
            a(x)|t_{\mu} w|^{r(x)} \log |t_{\mu}w|dx
            +\mu \int_{\Omega} c(x)|t_\mu w|^{\eta(x)}dx\\
            &\geq\displaystyle
            \int_{\Omega}b(x)t_{\mu}^{\alpha(x)}|w|^{\alpha(x)}dx
            -\mu \displaystyle\int_{ \Omega\cap\{ |t_{\mu}w(x)|<1\}}
            a(x)|t_{\mu} w|^{r^-} |\log |t_{\mu}w|\,|dx
            +\mu \int_{\Omega} c(x)|t_\mu w|^{\eta(x)}dx\\
            &\geq\displaystyle
            \int_{\Omega}b(x)t_{\mu}^{\alpha(x)}|w|^{\alpha(x)}dx
            -\mu C_{r^--\eta^+}\displaystyle \int_{\Omega\cap\{ |t_{\mu}w(x)|<1\}}a(x)|t_{\mu}w|^{\eta^+}dx
            +\mu \int_{\Omega} c(x)|t_\mu w|^{\eta(x)}dx\\
            &\geq \displaystyle
            \int_{\Omega}b(x)t_{\mu}^{\alpha(x)}|w|^{\alpha(x)}dx + \mu \displaystyle
            \int_{\Omega} [c(x)-C_{r^--\eta^+}a(x)]|t_\mu w|^{\eta(x)}dx \\
            &\geq \displaystyle
            \int_{\Omega}b(x)t_{\mu}^{\alpha(x)}|w|^{\alpha(x)}dx.
        \end{aligned}
        \end{equation}
        We claim that the sequence $(t_{\mu})_{\mu>0}$ is bounded in $\mathbb R^+$.
        Otherwise, denoting by
        $$\Theta=\left\{\mu>0:\,\,\rho(t_{\mu}w)\geq 1\right\},$$
         it follows from
        Proposition \ref{norm-modular} that
        \begin{align}\label{mueq3}
            \rho(t_{\mu}w)\leq t_{\mu}^{p_{2}^{+}} \|w\|^{p_{2}^{+}}, \ \
            \mbox{for any} \ \  \mu \in \Theta.
        \end{align}
        Thus, by \eqref{mueq1} and \eqref{mueq3} and for any $\mu\in \Theta $, we get
        \begin{align*}
            t_{\mu}^{p_{2}^{+}}
            \|w\|^{p_{2}^{+}} \geq t_{\mu}^{\alpha^-}\displaystyle
            \int_{\Omega}b(x)|w|^{\alpha(x)}dx,
        \end{align*}
        which gives a contradiction, since
        $p_{2}^{+}<\alpha^{-}$ by $(H_{10})$. This concludes the proof of the claim.

        Fix a sequence $(\mu_{n})_{n\geq 1}\subset \mathbb{R}^{+}$ such that
        $\mu_{n}\rightarrow\infty$ as $n\rightarrow\infty$. Evidently
        $(t_{\mu_{n}})_{n\geq 1}$ is bounded. Thus, there exist a constant
        $t_{0}>0$ and a subsequence of $(\mu_{n})_{n\geq1}$, still denoted by
        $(\mu_{n})_{n\geq1}$, such that $t_{\mu_{n}}\rightarrow t_{0}$ as
        $n\rightarrow\infty$. In view of Lemma \ref{s+}, we know that the
        sequence $(\rho(t_{\mu_{n}}w))_{n\geq1}$ is still bounded, and so according
        to \eqref{mueq1}, there exists $D>0$ such that
        \begin{align}\label{mueq4}
            \displaystyle
            \int_{\Omega}b(x)|t_{\mu_{n}}w|^{\alpha(x)}dx+\mu_n
            \displaystyle \int_{\Omega} a(x)|t_{\mu_{n}}w|^{r(x)} \log
            |t_{\mu_{n}} w|dx+\mu_n \displaystyle \int_{\Omega}
            c(x)|t_{\mu_{n}}w|^{\eta(x)}dx\leq D,
        \end{align}
        for any $n\in\mathbb N$.
        We assert that $t_{0}=0$.  Otherwise, by  dominated convergence theorem, we get
        \begin{align*}
            \displaystyle \lim_{n\rightarrow \infty}\Big[\displaystyle
            \int_{\Omega} a(x)|t_{\mu_{n}}w|^{r(x)} \log |t_{\mu_{n}}
            w|dx&+\displaystyle \int_{\Omega}
            c(x)|t_{\mu_{n}}w|^{\eta(x)}dx\Big]=\displaystyle \int_{\Omega}
            a(x)|t_{0}w|^{r(x)} \log |t_{0} w|dx\\
            &+\displaystyle \int_{\Omega} c(x)|t_{0}w|^{\eta(x)} dx.
        \end{align*}
        Also, we observe that using Lemma
        \ref{log-in} with $\sigma=r^--\eta^+$ and condition $(H_{10})$ again, as done in \eqref{mueq1}, we have
            \begin{equation*}
            \begin{aligned}
            \displaystyle
             \int_{\Omega}& a(x)|t_0w|^{r(x)} \log
            |t_0 w|dx+ \displaystyle \int_{\Omega}
            c(x)|t_0w|^{\eta(x)}dx\\
            &=\displaystyle \int_{\Omega\cap\{ |t_0w(x)|>1\}}
            a(x)|t_0 w|^{r(x)} \log |t_0w|dx\\
            &\qquad+\displaystyle\int_{ \Omega\cap\{ |t_0w(x)|\leq1\}}
            a(x)|t_0 w|^{r(x)} \log |t_0w|dx
            +\int_{\Omega} c(x)|t_0 w|^{\eta(x)}dx  \\
            &\geq\displaystyle \int_{\Omega\cap\{ |t_0w(x)|>1\}}
            a(x)|t_0 w|^{r^-} \log |t_0w|dx>0.
        \end{aligned}
        \end{equation*}
        Therefore, recalling that  $\mu_n \rightarrow \infty$, we infer
        that
        $$ \displaystyle
        \int_{\Omega}b(x)t_{\mu_{n}}^{\alpha(x)}|w|^{\alpha(x)}dx+\mu_{n}
        \left[\displaystyle \int_{\Omega} a(x)|t_{\mu_{n}}w|^{r(x)} \log
        |t_{\mu_{n}} w|dx+\displaystyle \int_{\Omega}
        c(x)|t_{\mu_{n}}w|^{\eta(x)}\right] \rightarrow \infty \ \
        \mbox{as} \ \ n\rightarrow \infty,$$
        which contradicts \eqref{mueq4}. Hence $t_{0}=0$ and $t_{\mu}\rightarrow 0$ when $\mu
        \rightarrow \infty$, since the
        sequence is arbitrary.

        Consider now the path $\nu_0(t)=tw$, $t\in[0,1]$, belonging to $\Gamma$. Then, Lemma \ref{geo2} gives
        \begin{equation}
        \begin{aligned}\label{vedi}
            0<c_{\mu}&\leq\max\limits_{t\in[0,1]}J_{\mu}(\nu_0(t))\leq
            J_{\mu}(t_{\mu}w)\\
            &\leq
            \rho(t_{\mu}w)+\mu\int_{\Omega}\frac{a(x)}{r^{2}(x)}|t_{\mu}w|^{r(x)}dx-\mu\displaystyle
            \int_{\Omega} \frac{a(x)}{r(x)}|t_{\mu}w|^{r(x)}
            \log|t_{\mu}w|dx-\mu\displaystyle \int_{\Omega}
            \frac{c(x)}{\eta(x)}|t_{\mu}w|^{\eta(x)}dx.
        \end{aligned}
        \end{equation}
        We claim that $(\mu t_\mu^{\eta^+})_{\mu>0}$ is bounded. Otherwise, there exists a sequence $(\mu_n t_{\mu_n}^{\eta^+})_{n\geq1}$ such that $\mu_n t_{\mu_n}^{\eta^+}\to\infty$ as $n\to\infty$. Thus, by \eqref{vedi} and considering that $t_{\mu_n}\to0$ as $n\to\infty$, we get
        $$
        \begin{aligned}
            0<\frac{c_{\mu_n}}{\mu_n t_{\mu_n}^{\eta^+}}\leq&\frac{\rho(t_{\mu_n}w)}{\mu_n t_{\mu_n}^{\eta^+}}+t_{\mu_n}^{r^--\eta^+}\int_{\Omega}\frac{a(x)}{r^{2}(x)}|w|^{r(x)}dx+\displaystyle
            t_{\mu_n}^{r^--\eta^+}\int_{\Omega} \frac{a(x)}{r(x)}|w|^{r(x)}
            |\log|t_{\mu_n}w|\,|dx\\
            &-\displaystyle \int_{\Omega}
            \frac{c(x)}{\eta(x)}|w|^{\eta(x)}dx
        \end{aligned}
        $$
        from which, by sending $n\to\infty$ and since $\eta^+<r^-$ by $(H_{10})$, we get
        $$
        0\leq-\displaystyle \int_{\Omega}
            c(x)|w|^{\eta(x)}dx<0
        $$
        that is the desired contradiction.
        Hence, $(\mu t_\mu^{\eta^+})_{\mu>0}$ is bounded so that by \eqref{vedi} again, we have
        $$
        \begin{aligned}
            0<c_{\mu}
            &\leq
            \rho(t_{\mu}w)+\mu t_\mu^{\eta^+}\int_{\Omega}\frac{a(x)}{r^{2}(x)}t_{\mu}^{r^--\eta^+}|w|^{r(x)}dx-\mu t_\mu^{\eta^+}\displaystyle
            \int_{\Omega} \frac{a(x)}{r(x)}t_{\mu}^{r^--\eta^+}|w|^{r(x)}
            |\log|t_{\mu}w|\,|dx.
        \end{aligned}
        $$
        Therefore, using the continuity of $\rho$ and using the fact that $w$ does not depend on $\mu$,
        we get that $c_{\mu}\rightarrow0$ as
        $\mu\rightarrow \infty$.\\
        This completes the proof of the lemma.
    \end{proof}

    Now we are ready to prove the Palais-Smale condition for $J_\mu$.
    \begin{lemma}\label{palais}
        Let the hypotheses in Theorem \ref{mainthm-3} hold.
        Then, there is $\mu_{0}>0$ such that for any $\mu>\mu_{0}$, the
        functional $J_{\mu}$ satisfies the Palais-Smale condition at the
        level $c_{\mu}>0,$ where $c_\mu $ is defined in \eqref{cmu}.
    \end{lemma}
    \begin{proof}
        Fix $\mu>0$ and let $(v_n)_{n\geq1}$ be a $(PS)_{c_\mu}$ sequence of the functional $J_\mu$, that is satisfying \eqref{pss}. We first show that $(v_n)_{n\geq1}$ is bounded in $X$. We argue by contradiction. Then, up to a subsequence, still
        denoted by $(v_n)_{n\geq1}$, we have
        $\lim\limits_{n\to\infty}\|v_n\|=\infty$ and so, $\|v_n\|\geq 1$ for $n\in\mathbb N$ sufficiently large.
     Let $\sigma_0>0$ be a positive constant  defined
    as  in assumption $(H_{10})$. Then by \eqref{pss} there exists $d>0$ such that, using Lemma \ref{log-in} and $(H_{10})$ we deduce as $n\to\infty$
        \begin{align}\label{ps1}
            & o_n(1)+c_{\mu}+d\|v_n\|\nonumber\\
            &\geq J_\mu(v_n) -\frac{1}{\sigma_0} \langle J_\mu'(v_n),v_n\rangle\nonumber\\
            &\geq\left(\frac{1}{2p_2^+}-\frac{1}{\sigma_0}\right)\rho(v_n)+\left(\frac{1}{\sigma_0}-\frac{1}{\al^-}\right)\int_\Om b(x)|v_n|^{\al(x)}dx
            +\mu \left(\frac{1}{\sigma_0}-\frac{1}{\eta^-}\right)\int_\Om c(x)|v_n|^{\eta(x)}dx\nonumber\\
            &\qquad+\mu \int_{\Omega}\left(\frac{1}{\sigma_0}
            -\frac{1}{r(x)}\right) a(x) | v_n|^{r(x)}\log|v_n(x)| dx + \mu \int_{\Omega}\frac{a(x)}{r(x)^2}| v_n(x)|^{r(x)} dx\nonumber\\
            &\geq
            \left(\frac{1}{2p_2^+}-\frac{1}{\sigma_0}\right)\rho(v_n)+\mu
            \int_{\Omega\cap\{|v_n(x)|<1\}}\left(\frac{1}{\sigma_0}
            -\frac{1}{r(x)}\right) a(x) | v_n|^{r(x)}\log|v_n(x)| dx \nonumber\\
            &\qquad+\mu \left(\frac{1}{\sigma_0}-\frac{1}{\eta^-}\right)\int_{\Omega\cap\{|v_n(x)|<1\}} c(x)|v_n|^{\eta(x)}dx\nonumber\\
            &\geq\left(\frac{1}{2p_2^+}-\frac{1}{\sigma_0}\right)\|v_n\|^{p_1^-}-\frac{\mu}{e(r^--\eta^+)}\left(\frac{1}{\sigma_0}
            -\frac{1}{r^+}\right)\int_{\Omega\cap\{ |v_n(x)|<1\}}a(x) |v_n|^{\eta(x)}dx
            \nonumber\\
            &\qquad+\mu \left(\frac{1}{\sigma_0}-\frac{1}{\eta^-}\right)\int_{\Omega\cap\{|v_n(x)|<1\}} c(x)|v_n|^{\eta(x)}dx\nonumber\\
            &\geq\left(\frac{1}{2p_2^+}-\frac{1}{\sigma_0}\right)\|v_n\|^{p_1^-}
        \end{align}
        which gives the desired contradiction, since $1<p_1^-<2p^{+}_{2}<\sigma_0$ by $(H_{10})$. Hence,
        $(v_n)_{n\geq1}$ is bounded in $X$.\\

        Since $X$ is reflexive, by $(H_7)$, Lemmas \ref{embd-X} and \ref{critical},  there exists a
        subsequence, still denoted by $(v_n)_{n\geq1}$, and $v_\mu\in X$ such
        that
        \begin{align}\label{ps2}
            \begin{split}
                &v_{n}\rightharpoonup v_{\mu} \ \ \text{ in }X, \quad \rho(v_{n})\rightarrow\beta_{\mu},\\
                &v_{n}\rightharpoonup v_{\mu}\text{ in }L^{\alpha(\cdot)}(\Omega),\quad\|b^{\frac{1}{\alpha(\cdot)}}(v_{n}-v_{\mu})\|_{L^{\alpha(\cdot)}(\Om)}\rightarrow \xi_{\mu},\\
                &v_{n}\rightarrow v_{\mu}\text{ in }L^{\eta(\cdot)}(\Omega) \ \
                \text{ and } \quad
                v_{n}(x)\rightarrow v_{\mu}(x)\text{ a.e. ~~in }\Omega,
            \end{split}
        \end{align}
        as $n\to\infty$, for any $\gamma\in C_+(\overline{\Omega})$ with $1<\gamma(x)<p_{2_{s_2}}^*(x)$ for any $x\in\overline{\Omega}$.
        Evidently, if $\beta_\mu=0$, we have $v_n\to0$ in $X$. Hence, let us
        suppose $\beta_\mu>0$.

        By Brezis-Lieb type lemmas in \cite{Fu} and \cite{zuo}, considering also $(H_7)$, we have
        \begin{equation}
        \begin{aligned}\label{BL}
        \rho(v_n)-\rho(v_n-v_\mu)&=\rho(v_\mu)+o_n(1)\\
            \int_{\Omega}b(x)\left(|v_{n}|^{\alpha(x)}-|v_{n}-v_\mu|^{\alpha(x)}\right)dx&=\int_{\Omega}b(x)|v_\mu|^{\alpha(x)}dx+o_n(1),
        \end{aligned}
        \end{equation}
        as $n\to\infty$.
        From this, arguing similarly to \eqref{ps1}, we get
        \begin{align}\label{ps4}
            c_{\mu}+o(1) &\geq
            \left(\frac{1}{\sigma_0}-\frac{1}{\al^-}\right)\int_\Om
            b(x)|v_n|^{\al(x)}dx+\mu\left(\frac{1}{\sigma_0}-\frac{1}{\eta^- }\right)\int_\Om
            c(x)|v_{n}|^{\eta(x)}dx \nonumber \\
            &\qquad+\mu \int_{\Omega\cap\{  |v_n(x)|<1\}}\left(\frac{1}{\sigma_0}
            -\frac{1}{r(x)}\right) a(x) | v_n|^{r(x)}\log|v_n(x)| dx \nonumber\\
            &\geq \left(\frac{1}{\sigma_0}-\frac{1}{\al^-}\right)\int_\Om
            b(x)|v_n|^{\al(x)}dx+\mu\left(\frac{1}{\sigma_0}-\frac{1}{\eta^-
            }\right)\int_{\Om\cap\{|v_n(x)|<1 \}}c(x)|v_{n}|^{\eta(x)}dx
            \nonumber\\
            &\qquad-\frac{\mu}{e(r^--\eta^+)}\left(\frac{1}{\sigma_0}
            -\frac{1}{r^+}\right)
            \int_{\Omega\cap\{ |v_n(x)|<1\}}a(x) | v_n|^{\eta(x)}
            dx\nonumber\\
            &\geq \left(\frac{1}{\sigma_0}-\frac{1}{\al^-}\right)\int_\Om
            b(x)|v_n|^{\al(x)}dx\nonumber\\
            &=\left(\frac{1}{\sigma_0}-\frac{1}{\al^-}\right)\left[\int_\Om
            b(x)|v_n-v_\mu|^{\al(x)}dx+\int_\Om
            b(x)|v_\mu|^{\al(x)}dx\right]\nonumber\\
            &\geq\left(\frac{1}{\sigma_0}-\frac{1}{\al^-}\right)\int_\Om
            b(x)|v_n-v_\mu|^{\al(x)}dx
        \end{align}
        as $n\rightarrow \infty$.
        For simplicity, let us set the operator $\mathcal L_i:X\to X^*$ such that
$$
\begin{aligned}
\langle\mathcal L_i(u),\varphi\rangle:=&\,
\frac{1}{2}\iint_{\R^{2N}\setminus(\mathcal{C}\Omega)^2}\frac{|u(x)-u(y)|^{p_i(x,y)-2}(u(x)-u(y))(\varphi(x)-\varphi(y))}{|x-y|^{N+s_i(x,y)p_i(x,y)}}dxdy\\
            &+ \int_{\Omega}|u|^{\overline{p}_i(x)-2}u\varphi dx
            +\int_{\mathcal{C}\Omega} \beta(x)|u|^{\overline{p}_i(x)-2}u \varphi dx,
\end{aligned}
$$
for any $u$, $\varphi\in X$ and for $i\in\{1,2\}$. Then for
$i\in\{1,2\}$, by \eqref{ps2} the sequence $(\mathcal
V^i_n)_{n\ge1}$, defined in $\mathbb R^{2N}\setminus(\mathcal
C\Omega)^2$ by
\begin{equation*}
(x,y)\mapsto\mathcal V^i_n(x,y)=\frac{|v_n(x)-v_n(y)|^{p_i(x,y)-2}[v_n(x)-v_n(y)]}{|x-y|^{(N+s_i(x,y)p_i(x,y))/p_i'(x,y)}},
\end{equation*}
is bounded in $L^{p_i'(x,y)}(\mathbb R^{2N})$, with $p_i'(x,y)=\frac {p_i(x,y)}{p_i(x,y)-1}$ as well as $\mathcal V^i_n(x,y)\to\mathcal V^i_\mu(x,y)$ point-wise a.e. in $\mathbb R^{2N}$, where
\begin{equation*}
\mathcal V^i_\mu(x,y)
=\frac{|v_\mu(x)-v_\mu(y)|^{p_i(x,y)-2}
[v_\mu(x)-v_\mu(y)]}{|x-y|^{(N+s_i(x,y)p_i(x,y))/p_i'(x,y)}}.
\end{equation*}
Thus, going if necessary to a further subsequence, we get that $\mathcal V^i_n\rightharpoonup
\mathcal V^i_\mu$ in $L^{p_i'(x,y)}(\mathbb R^{2N})$ as $n\to\infty$.
Similarly, we have $|v_n|^{\overline{p}_i(x)-2}v_n\rightharpoonup |v_\mu|^{\overline{p}_i(x)-2}v_\mu$ in $L^{\overline{p}_i'(x)}(\Omega)$ and
$|v_n|^{\overline{p}_i(x)-2}v_n\rightharpoonup |v_\mu|^{\overline{p}_i(x)-2}v_\mu$ in $L^{\overline{p}_i'(x)}(\mathcal C\Omega)$, as $n\to\infty$. Thus, we obtain that
\begin{equation}\label{serve}
\langle\mathcal L_i(v_n),\varphi\rangle\to\langle\mathcal L_i(v_\mu),\varphi\rangle\quad\mbox{as }n\to\infty,
\end{equation}
for any $\varphi\in X$.

        Now, let $\sigma\in\left(0,(p_{2_{s_2}}^*)^--r^+\right)$. By Lemmas \ref{embd-X} and \ref{log-in} and using H\"older inequality, we obtain
        \begin{align}\label{ps5}
            &\int_{\Omega} a(x) |v_n|^{r(x)-1} |\log (|v_n|) |v_n-v_{\mu}||dx\\
            &= \int_{\Omega\cap\{ |v_n|<1\}} a(x) |v_n|^{r(x)-1} |\log|
            (|v_n|) |v_n-v_{\mu}|dx +\int_{\Omega\cap\{ |v_n|>1\}}
            a(x) |v_n|^{r(x)-1} \log
            (|v_n|) |v_n-v_{\mu}|dx \nonumber\\
            &\leq \frac{\|a\|_{L^\infty(\Om)}}{e(r^--1)} \|v_n-v_{\mu}\|_{L^1(\Omega)}+
            \frac{\|a\|_{L^\infty(\Om)}}{e\sigma}\int_{\Omega}
            |v_n|^{r^++\sigma-1} |v_n-v_{\mu}|dx \nonumber\\
            & \leq \frac{\|a\|_{L^\infty(\Om)}}{e(r^--1)} \|v_n-v_{\mu}\|_{L^1(\Omega)}+ \frac{\|a\|_{L^\infty(\Om)}}{e\sigma}
            \||v_n|^{r^++\sigma-1}\|_{L^{\frac{r^++\sigma}{r^++\sigma-1}}(\Om)}\|v_n-v_{\mu}\|_{L^{r^++\sigma}(\Om)} \nonumber\\
            & \rightarrow 0 \ \ \mbox{as} \ \ n \rightarrow \infty. \nonumber
        \end{align}
        Thus, by \eqref{ps2}, \eqref{serve} and \eqref{ps5} we have as $n\rightarrow
        \infty$
        \begin{align*}
            o(1)&= \langle J'_{\mu}(v_{n}), v_{n}-v_{\mu}\rangle\\
            &=\langle\mathcal L_1(v_n),(v_n-v_\mu)\rangle+\langle\mathcal L_2(v_n),(v_n-v_\mu)\rangle- \int_{\Omega}
            b(x)|v_n|^{\alpha(x)-2}v_{n}(v_n-v_{\mu})dx \nonumber\\
            &= \rho(v_n)-\rho(v_{\mu})- \int_{\Omega} b(x)\left(
            |v_n|^{\alpha(x)}-|v_{\mu}|^{\alpha(x)}\right)dx+o(1) \nonumber.
        \end{align*}
        From this and by \eqref{BL}, we get
        \begin{equation}\label{ps7}
            \lim_{n\rightarrow \infty}\rho(v_n-v_{\mu})= \lim_{n
                \rightarrow \infty} \int_{\Omega} b(x)|v_n-v_{\mu}|^{\alpha(x)}dx.
        \end{equation}
        Then by Proposition \ref{norm-modular}, Lemma \ref{critical}, \eqref{ps7} and $(H_7)$, we infer that
        \begin{equation}\label{ps8}
            \max\{\xi_\mu^{\alpha^+},\xi_\mu^{\alpha^-}\}\geq
            \min\{(C_{b,S}\xi_\mu)^{p_{2}^+},(C_{b,S}\xi_\mu)^{p_{1}^-}\},
        \end{equation}
        where $C_{b,S}=S^{-1}\max\left\{\|b\|_{L^\infty(\Omega)}^{1/\alpha^+},\|b\|_{L^\infty(\Omega)}^{1/\alpha^-}\right\}>0$.
        By combining \eqref{ps4} and \eqref{ps8}, we can deduce that
        $$c_{\mu}\geq \left(\frac{1}{\sigma}-\frac{1}{\alpha^{-}}\right)\times
        \begin{cases}
            C_{b,S}^{\frac{\alpha^{+}p_{2}^{+}}{\alpha^{-}-p_{2}^{+}}}, \ \ \mbox{if}
            \ \ \xi_\mu \in
            (0,1) \ \ \mbox{and} \ \ C_{b,S}\xi_\mu \in (0,1)\\
            C_{b,S}^{\frac{\alpha^{-}p_{1}^{-}}{\alpha^{+}-p_{1}^{-}}}, \ \ \mbox{if} \ \ \xi_\mu \in
            (1,\infty) \ \ \mbox{and} \ \ C_{b,S}\xi_\mu \in (1,\infty)\\
            C_{b,S}^{\frac{\alpha^{-}p_{2}^{+}}{\alpha^{+}-p_{2}^{+}}}, \ \ \mbox{if}
            \ \ \xi_\mu \in
            (1,\infty) \ \ \mbox{and} \ \ C_{b,S}\xi_\mu \in (0,1)\\
            C_{b,S}^{\frac{\alpha^{+}p_{1}^{-}}{\alpha^{-}-p_{1}^{-}}}, \ \ \mbox{if}
            \ \ \xi_\mu \in (0,1) \ \ \mbox{and} \ \ C_{b,S}\xi_\mu \in (1,\infty),
        \end{cases}
        $$
        which is impossible by Lemma \ref{mu}. Hence $\xi_\mu=0$. Therefore
        by \eqref{ps7} we conclude that $v_n \rightarrow v_\mu$ in $X$ as $n\to\infty$.\\
        This completes the proof.
    \end{proof}

    \noindent
    \textbf{Proof of Theorem \ref{mainthm-3}:} By Lemmas
    \ref{geo2} and \ref{palais}, we can conclude that there is $\mu_0>0$
    such that, for every $\mu>\mu_0$, the functional $J_\mu$ fulfills
    all conditions of the classical mountain Pass theorem. Then, for any
    $\mu>\mu_0$, problem \eqref{eq1} admits a weak solution $v_\mu\in
    X$. Evidently, $v_\mu \neq 0$ due to the fact that
    $J_\mu(v_\mu)=c_{\mu}>0.$

We just have to verify the asymptotic behavior \eqref{asym}. Thanks to Proposition \ref{norm-modular}, it is equivalent to prove that $\lim\limits_{\mu\to\infty}\rho(v_\mu)=0$. Suppose by contradiction that
$\limsup\limits_{\mu\to\infty}\rho(v_\mu)=\upsilon>0$. Hence there is a sequence $n\mapsto \mu_n\uparrow\infty$ such that $\rho(v_{\mu_n})\to \upsilon$ as $n\to\infty$. Then, arguing as in \eqref{ps1} we get
$$
            c_{\mu_n}\geq
            \left(\frac{1}{2p_2^+}-\frac{1}{\sigma_0}\right)\rho(v_{\mu_n}).
$$
By sending $n\to\infty$ and considering Lemma \ref{mu}, we conclude
$$
0\geq\left(\frac{1}{2p_2^+}-\frac{1}{\sigma_0}\right)\upsilon>0
$$
which gives the desired contradiction.
    This ends the proof of our main result. \\

    \textbf{Concluding remarks, perspectives, and some open questions}\\
    \begin{itemize}
        \item The methods developed in this paper can be extended to more
        general variational integrals. We mainly refer to energy functionals
        associated to fractional Orlicz operators of the type $(-\Delta
        )_{g}^{s}$, which extend the standard fractional $p$-Laplacian and
        different to the fractional $p(\cdot)$-Laplacian. These operators have
        been introduced by A.M. Salort et al. \cite{salort}.
        \item Condition $(H_{10})$ played a key role in the proof of the
        existence of solution for the critical problem \eqref{eq1}. In
        particular, the main aim of the presence of the term $c(x)|u|^{\eta(x)-2}u$
        is to control the negative part of the logarithmic term. Thus, it is
        a natural question to study what happens if $c(x)=0$.
        \item We believe that a valuable research direction is to
        generalize the abstract approach developed in this paper to the whole space $\mathbb{R}^{N}$.
        \item In a forthcoming paper, we will be interested in the study of
        new classes of nonlinear boundary value logarithmic problems involving the
        double phase operator of type  as in \cite{Liu}.
    \end{itemize}

    \section*{Acknowledgments}

A. Fiscella is member of {\em Gruppo Nazionale per l'Analisi Ma\-te\-ma\-ti\-ca, la Probabilit\`a e le loro Applicazioni} (GNAMPA)
of the {\em Istituto Nazionale di Alta Matematica} (INdAM).
A. Fiscella is partially supported by INdAM-GNAMPA project titled {\em Equazioni alle derivate parziali: problemi e modelli} and by the FAPESP Thematic Project titled {\em Systems and partial differential equations} (2019/02512-5).

\end{document}